\documentclass[11pt,letterpaper, reqno]{amsart}

\usepackage[english]{babel}

\usepackage[margin = 1.in]{geometry}

\usepackage[numbers]{natbib} % for citet citations etc.
\usepackage{amsmath}
\usepackage{amssymb}
\usepackage{amsthm}
\usepackage{graphicx}
\usepackage{comment}
\usepackage{mathtools}
\usepackage{nicefrac}
\usepackage{xcolor}
\usepackage{overpic}
\usepackage[normalem]{ulem}
\usepackage{stmaryrd} % Where \llbracket and \rrbracket are defined

\usepackage{algorithm}
\usepackage{algpseudocode}
\usepackage{listings}
\usepackage{xcolor}

\usepackage{subcaption}

\usepackage{enumitem}
\setlist[enumerate]{label=({\roman*})}

\newtheorem{lemma}{Lemma}
\newtheorem{corollary}{Corollary}
\newtheorem{proposition}{Proposition}
\newtheorem{theorem}{Theorem}
\newtheorem{definition}{Definition}
\newtheorem{example}{Example}

\newtheorem{remark}{Remark}

\newcommand{\figref}[1]{Figure~\ref{#1}}
\newcommand{\corref}[1]{Corollary~\ref{#1}}
\newcommand{\defref}[1]{Definition~\ref{#1}}
\newcommand{\teqref}[1]{Equation~\eqref{#1}}

\newcommand{\secref}[1]{Section~\ref{#1}}
\newcommand{\appref}[1]{Appendix~\ref{#1}}
\newcommand{\thmref}[1]{Theorem~\ref{#1}}
\newcommand{\lemref}[1]{Lemma~\ref{#1}}
\newcommand{\propref}[1]{Proposition~\ref{#1}}

\newcommand{\tblref}[1]{Table~\ref{#1}}

\newcommand{\exref}[1]{Example~\ref{#1}}

\def\ie{\emph{i.e.}}
\def\eg{\emph{e.g.}}

\def\resp{\emph{resp.}}

\def\const{\operatorname{const}}

\def\inv{\operatorname{inv}}
\def\cycT{\operatorname{Cyc}(\cT)}
\def\area{\operatorname{area}(\Delta)} % \triangle was already taken as the set of triangles
\def\spann{\operatorname{span}}

\def\tPsi{\tilde{\Psi}}
\def\ttPsi{\tilde{\tilde{\Psi}}}
\def\RR{\mathbb{R}}
\def\NN{\mathbb{N}}

\def\bA{{\bf A}}
\def\bB{{\bf B}}
\def\bC{{\bf C}}

\def\bI{{\bf I}}

\def\bL{{\bf L}}
\def\bM{{\bf M}}

\def\bP{{\bf P}}
\def\bQ{{\bf Q}}
\def\bR{{\bf R}}
\def\bS{{\bf S}}
\def\bT{{\bf T}}

\def\bV{{\bf V}}

\def\bX{{\bf X}}

\def\ba{{\bf a}}
\def\bb{{\bf b}}
\def\bc{{\bf c}}

\def\bPhi{{\bf \Phi}}
\def\bPsi{{\bf \Psi}}
\def\bPhiA{{\bf \Phi}_{\operatorname{A}}}
\def\bPhiN{{\bf \Phi}_{\operatorname{N}}}
\def\bPhiS{{\bf \Phi}_{\operatorname{S}}}

\def\bPhiAinv{{\bf \Phi}_{\operatorname{A}}^{-1}}
\def\bPhiNinv{{\bf \Phi}_{\operatorname{N}}^{-1}}

\def\bPhiID{{\bf \Phi}_{\operatorname{Id}}}
\def\bomega{{\bf \omega}}
\def\bPhiR{{\bf \Phi}_{\operatorname{R}}}
\def\bomega{{\bf \omega}}

\def\tN{t_{\operatorname{N}}}
\def\tA{t_{\operatorname{A}}}

\def\cT{{\mathcal{T}}}

\newtheorem{deflemma}[definition]{Definition \& Lemma}

\newcommand{\nmlz}[1]{{\Sigma_{#1}}}

\usepackage[colorlinks=true, allcolors=blue]{hyperref}

\title{
(Semi-)Invariant Curves from Centers of Triangle Families
}

\author{Klara Mundilova} 
\address{EPFL, Station 14, 1015 Lausanne, Vaud, Switzerland}
\curraddr{}
\email{klara.mundilova@epfl.ch}
\thanks{}

\author{Oliver Gross} 
\address{University of California San Diego, 9500 Gilman Dr, La Jolla, CA 92093, USA}
\curraddr{}
\email{ogross@ucsd.edu}
\thanks{}

\date{\today}

\begin{document}

\begin{abstract}
We study curves obtained by tracing triangle centers within special families of triangles, focusing on centers and families that yield (semi-)invariant triangle curves, meaning that varying the initial triangle changes the loci only by an affine transformation. We identify four two-parameter families of triangle centers that are semi-invariant and determine which are invariant, in the sense that the resulting curves for different initial triangles are related by a similarity transformation. We further observe that these centers, when combined with the aliquot triangle family, yield sheared Maclaurin trisectrices, whereas the nedian triangle family yields Lima\c{c}on trisectrices.

\medskip\noindent\textbf{Keywords:} triangle centers; loci of triangle centers; affine invariance; similarity invariance;  Maclaurin trisectrix; Lima\c{c}on trisectrix; plane algebraic curves

\medskip\noindent\textbf{MSC 2020:} 51M15, 51M04, 14H50, 51N10, 51N20
\end{abstract}
\maketitle
\begin{flushright}
        {\bf \emph{``To me, this was weird.''}}\\
    D. Hofstadter\footnote{Taken from the foreword in \cite{Kimberling:1998:TCC}.}
\end{flushright}

\section{Introduction} 
Given a triangle, a number of quantities and objects can be associated with it, arguably the simplest being its side lengths and interior angles. In addition, there are selected points such as the centroid, the circumcenter, and the Fermat point, which are well-known representatives of so-called \emph{triangle centers}. The study of triangle centers can be traced back to the ancient Greeks and, pioneered by Kimberling~\cite{Kimberling:1998:TCC, Kimberling:2025:ETC}, it also enjoys attention in contemporary research%triangle geometry
~\cite{Abu:2005:TCL, Narboux:2016:TCV}. Notably, the study of selected configurations of triangle centers led to the discovery of special algebraic curves associated to triangles~\cite{Abu:2007:ACA, Cundy:1995:SCC}. For example, the Balaton curve---which contains a triangles orthocenter, circumcenter, in-center, Toricelli's point, and first isodynamic point---is algebraic for certain triangles~\cite{Dirnbock:2003:Curves}. 

Other lines of work examine curves associated to triangles obtained by tracing points (or triangle centers) in a one-parameter family of triangles. Examples of such families include poristic triangles~\cite{Odehnal:2011:PLT} and triangles obtained by linearly displacing one of the triangle's vertices~\cite{Jurkin:2022:LOC, Kodrnja:2023:LCT}.

Aliquot and nedian triangles are notable examples of triangle families in which the centroids are concurrent \cite{Satterly:1956:NNT}. Kodrnja and Koncul~\cite{Kodrnja:2017:LVN} describe the loci of the vertices of nedian triangles, while Mundilova~\cite{Mundilova:2024:MTT}  demonstrates that the first isogonic and isodynamic centers of aliquot triangles trace an algebraic curve, the \emph{Maclaurin trisectrix}, regardless of the initial triangle.

Inspired by the functional definition of triangle centers, we introduce the concept of $\bPsi$-triangle families, which includes the aliquot and nedian families as prominent examples. The central theme of this work is the study of curves traced by triangle centers as the underlying triangle varies within such a family. We identify triangle centers, and combinations thereof that give rise to curves with strong geometric rigidity. Under changes of the reference triangle, they transform merely by affine or even similarity transformations.

To formalize this notion, we distinguish between two types of rigidity. A triangle curve is called \emph{semi-invariant} if, under a change of the reference triangle, the curve traced by a given triangle center changes only by an affine transformation, namely translations, rotations, scaling, or sheering. A curve is called \emph{invariant} if it is semi-invariant without allowing sheering.

\subsection{Outline}
This paper is organized as follows. In \secref{sec:triangleCenters}, we review the definition of triangle centers and summarize a selection of properties that will be used throughout the paper. In \secref{sec:PsiTriangleFamilies}, we introduce $\bPsi$-triangle families and show that they form a geometric structure of independent interest. In particular, we prove that the set of triangle families, equipped with concatenation, carries the structure of an Abelian group.

Our main results are presented in \secref{sec:trianglecurves}. We first study local and global properties of triangle curves in \secref{sec:propertiesOfTriangleCurves}, with emphasis on semi-invariance and invariance. We show that the property of a triangle center generating a semi-invariant curve with respect to the aliquot family extends to a broad class of $\bPsi$-triangle families. We also prove that invariance is preserved under concatenation of families.

Building on work of Mundilova~\cite{Mundilova:2024:MTT}, who showed that two triangle centers generate Maclaurin trisectrices when combined with the aliquot family, we substantially extend this result in \secref{sec:mainresult}. We identify four two-parameter families of semi-invariant triangle centers with respect to the class of decomposable triangle families.  Within these families, we characterize those triangle centers that give rise to invariant curves. We show that in the special case of aliquot or nedian triangle families, the resulting curves are sheared Maclaurin or Lima\c{c}on trisectrices, respectively. 

Finally, in~\secref{sec:inverseProblem}, we address an inverse problem. For the semi-invariant curve-generating triangle centers identified in~\secref{sec:mainresult}, we determine the $\bPsi$-triangle family that produces a prescribed admissible target curve.

Before turning to the study of triangle centers, we fix notation and conventions that will be used throughout the paper. For readability, we present the main concepts and results in the body of the paper. Unless stated otherwise, proofs are deferred to the appendix.

\subsection{Notation}
Any distinct three points \(\bA,\bB,\bC\in\RR^2\) uniquely define a triangle \(\Delta\) with edges \(\ba\), \(\bb\), and \(\bc\) opposite the respective vertices, whose lengths we denote by $a, b, c\geq 0$ respectively. Additionally, we follow the convention to label the angles incident to the vertices $\bA$, $\bB$, and $\bC$ by $\alpha$, $\beta$, and $\gamma$, respectively (see \figref{fig:notation}). 
In the following, we consider only triangles \(\Delta\) that are non-degenerate, \ie, the three vertices are assumed to be in general position. 
We denote the set of all such triangles by
\[\triangle\coloneqq\{\Delta\mid \bA,\bB,\bC\in\RR^2 \text{ are distinct and non-collinear points} \}.\]
We can equivalently describe the set of non-degenerate triangles in terms of their side-lengths, \ie, triples \((a,b,c)\in\RR^3_{>0}\) such that \(a<  b+c\), \(b < c+a\) and \(c < a+b\). For brevity, and by slight abuse of notation we will denote the set of such triples also by \(\triangle\). Using \emph{Heron's formula}, we express the area of the triangle as
$$
\area = \frac{1}{4}\sqrt{(a + b + c)(-a + b + c)(a - b + c)(a + b -c)}.
$$ 
Note that $\area\neq 0$ for $\Delta \in \triangle$.

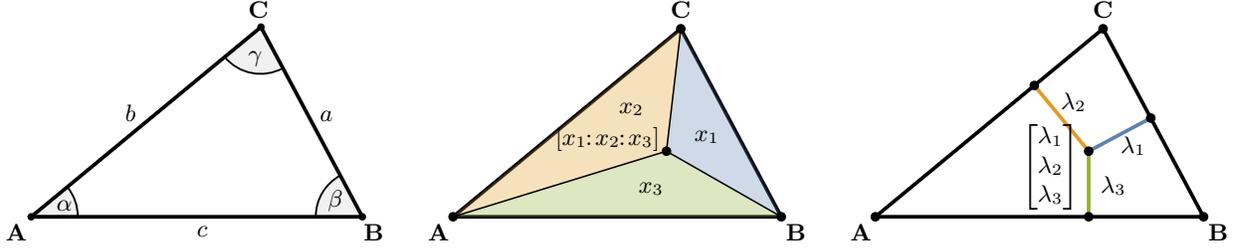
\begin{figure}[!t]
\begin{footnotesize}
\begin{overpic}[width=0.32\textwidth]{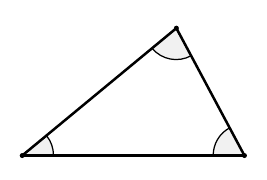}
\put(2,3){$\bA$}
\put(92,3){$\bB$}
\put(63,59){$\bC$}
\put(50,4){$c$}
\put(81,33){$a$}
\put(32,33){$b$}
\put(14.75,10.5){$\alpha$}
\put(83,11.5){$\beta$}
\put(63,48){$\gamma$}
\end{overpic}
\hfill
\begin{overpic}[width=0.32\textwidth]{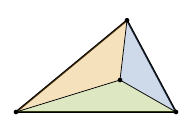}
\put(2,3){$\bA$}
\put(92,3){$\bB$}
\put(63,59){$\bC$}
\put(34,27){$[x_1\!\!:\!x_2\!\!:\!x_3]$}
\put(69,28){$x_1$}
\put(50,35){$x_2$}
\put(55,15){$x_3$}
\end{overpic}
\hfill
\begin{overpic}[width=0.32\textwidth]{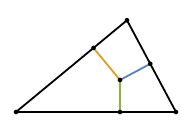}
\put(2,3){$\bA$}
\put(92,3){$\bB$}
\put(63,59){$\bC$}
\put(45,20){$\begin{bmatrix}\lambda_1\\\lambda_2\\\lambda_3\end{bmatrix}$}
\put(65,15){$\lambda_3$}
\put(70,25){$\lambda_1$}
\put(55,36){$\lambda_2$}
\end{overpic}%
\end{footnotesize}%
\caption{Illustration of the notation used in this paper, highlighting geometric interpretations of barycentric coordinates (center) and trilinear coordinates (right).
}\label{fig:notation}
\end{figure}

To keep the in parts lengthy equations as concise as possible we employ two sets of coordinate systems throughout the paper, which both specify points in the plane in terms of the vertex positions \(\bA,\bB,\bC\in\RR^2\). First, we consider \emph{homogeneous barycentric coordinates} \([\lambda_1 : \lambda_2:\lambda_3]^T\), denoted as column vectors, which determine the position of a point \(\bP\in\RR^2\) by
\[\bP =\frac{1}{\lambda_1 + \lambda_2 + \lambda_3}\left( \lambda_1\bA + \lambda_2\bB + \lambda_3\bC\right).\]
They can be \emph{normalized} by asking that \(\lambda_1+\lambda_2+\lambda_3=1\).
In contrast, \emph{homogeneous trilinear coordinates} \([x_1:x_2:x_3]\), denoted as row vectors, determine the same point \(\bP\) as 
$$
\bP = \frac{1}{a \lambda_1 + b \lambda_2 + c \lambda_3}\left( ax_1 \bA + bx_2 \bB + cx_3 \bC\right).
$$
Whenever it is reasonably defined, the conversion between coordinates of these two sets of coordinate systems is achieved by division or multiplication with the triangle's side lengths: \[[x_1:x_2:x_3] = 
\begin{bmatrix}
\nicefrac{\lambda_1}{a}\\ \nicefrac{\lambda_2}{b}\\ \nicefrac{\lambda_3}{c} 
\end{bmatrix},\quad \text{\resp}\quad  \begin{bmatrix} \lambda_1\\\lambda_2\\ \lambda_3\end{bmatrix} = [ ax_1 :bx_2 :cx_3].\]
Consequently, division of the expressions \(ax_1, bx_2, cx_3\) by \(ax_1+bx_2+cx_3\) yields the conversion from trilinear to normalized barycentric coordinates.

\section{Triangle Centers}\label{sec:triangleCenters}

\subsection{Definition and preliminaries }

The following definition of a triangle center is adapted from Kimberling~\cite{Kimberling:1998:TCC}.
\begin{definition}
    \label{def:TriangleCenter}
A function \(\psi\colon\triangle\to\RR\), \(\psi\not\equiv 0\), is said to be a \emph{triangle center function} if it is
\begin{enumerate}
    \item \emph{homogeneous}, \ie, there exists some constant $r\in\RR$ such that for all $t>0$
we have \begin{equation}
            \label{eq:Homogeneity}
            \psi(ta, tb, tc) = t^r \psi(a, b, c).
        \end{equation} 
        In this paper, we will use the notation $\deg(\psi) = r$ to indicate that $\psi$ has homogeneity $r$.
\item \emph{bi-symmetric} in the second and third variables, \ie, for all \(a,b,c\in\triangle\), we
have \begin{equation}
        \label{eq:Bisymmetry}
        \psi(a, b, c) = \psi(a, c, b).
     \end{equation}
    \end{enumerate}
The set of all triangle center functions is denoted by \(\cT\). Each \(\psi\in\cT\) defines a \emph{triangle center} of \((a,b,c)\in\triangle\) with trilinear  coordinates 
$$
\bX_{\psi} = [\psi(a,b,c) : \psi(b,c,a) : \psi(c,a,b)].
$$
\end{definition}
While Kimberling only considers positive $r$,  our definition also allows for negative $r$, as the degree of homogeneity of any triangle center function can be increased, for example, by multiplying by powers of $abc$. Additionally, note that the definition of a triangle center is equivariant under similarity transformations, \ie, under translations, rotations, reflections, and uniform scalings. That is, the center of the transformed triangle is the same point as the transformed center of the original triangle. To describe triangle center functions that correspond to the same triangle centers, we introduce: 
\begin{definition}\label{def:cycT}
We denote by $\cycT$ the subset of functions in $\cT$ that are nowhere vanishing, \ie, $f(a,b,c) \neq 0$ for all $(a,b,c)\in \triangle$ and \emph{invariant under cyclic permutations}, that is, $f(a,b,c) \in \cycT$ satisfies
$$
f(a,b,c) = f(b,c,a) = f(c,a,b).
$$
\end{definition}
The following \lemref{lem:equivCenterFunctions}, which is an immediate consequence of \defref{def:cycT}, clarifies when two triangle center functions represent the same triangle center, namely when they differ by a cyclic factor. 
\begin{lemma}\label{lem:equivCenterFunctions}
    Two triangle center functions $\psi_0(a,b,c)$ and $\psi_1(a,b,c)$ describe the same triangle center if 
    $$
   \psi_1(a,b,c) =  f(a,b,c)\, \psi_0(a,b,c),
    $$
    for some $f(a,b,c) \in \cycT$.
We will indicate that two triangle center functions $\psi_0$ and $\psi_1$ describe the same triangle center by $\psi_0 \cong \psi_1$. 
\end{lemma}

\begin{remark}
Since \(abc\in\cycT\), we have \(\psi \cong abc\,\psi\) for all \(\psi\in\cT\). Thus, in \defref{def:TriangleCenter}, homogeneity is the essential requirement, whereas the particular degree can be shifted by multiplying by powers of \(abc\) and therefore has no geometric significance (on nondegenerate triangles).
\end{remark}

Triangle centers are tied together by several natural dualities, and among the most useful is the passage to an isogonal conjugate. In this setting, isogonal conjugation acts on the data defining a center by inverting the associated triangle center function, producing a new center that is paired with the original one.
\begin{definition}
    Given a triangle center $\bX_{\psi}$, its \emph{isogonal conjugate} $\bX_{\psi^{-1}}$ is a triangle center that corresponds to the triangle center function $\psi^{-1}(a,b,c) \coloneqq (\psi(a,b,c))^{-1}$. 
\end{definition}

The set of triangle center functions \(\cT\) is known to be a continuum. A catalog of more than 60,000 distinct, numbered triangle centers is available in the \emph{Encyclopedia of Triangle Centers}~\cite{Kimberling:2025:ETC}. A selection of centers addressed in the present work along with their isogonal conjugates, is presented in~\tblref{tab:trianglecenters}.

\begin{table}[t]
\centering
    \begin{tabular}{c | l | c | c | c}
        $\bX_i$ & Name & $\phi_i(a,b,c)$ & traceable & $\bX_i^{-1}$ \\
        \hline
        $\bX_1$ & Incenter & $1$ & yes & $\bX_1$\\
        $\bX_2$ & Centroid & $a^{-1}$ & yes &$\bX_6$\\
        $\bX_3$ & Circumcenter & $a(-a^2 + b^2+c^2)$ & yes &$\bX_4$\\
        $\bX_6$ & Symmedian point & $a$ & yes &$\bX_2$ \\
        $\bX_{11}$ & Feuerbach point & $bc(-a+b+c)(b-c)^2$ & no & $\bX_{59}$ \\
        $\bX_{13}$ & First isogonic center & $\csc\left(\alpha - \nicefrac{\pi}{3}\right)$ &yes & $\bX_{15}$ \\
        $\bX_{14}$ & Second isogonic center & $\csc\left(\alpha + \nicefrac{\pi}{3}\right)$ &yes & $\bX_{16}$\\
        $\bX_{15}$ & First isodynamic center & $\sin\left(\alpha - \nicefrac{\pi}{3}\right)$ &yes & $\bX_{13}$ \\
        $\bX_{16}$ & Second isodynamic center & $\sin\left(\alpha + \nicefrac{\pi}{3}\right)$ &yes & $\bX_{14}$\\
    \end{tabular}
    \caption{Selection of centers and their generating functions.}\label{tab:trianglecenters}
\end{table}

\paragraph{Convention} In this paper, we use $\psi_j$ to denote a generic triangle center function and $\phi_i$ to refer to the triangle center function corresponding to a specific, numbered triangle center (according to~\cite{Kimberling:2025:ETC}). That is, $\phi_i$ represents the triangle center function for the $i$-th triangle center $\bX_i$. For the isogonal conjugate of a triangle center $\bX_{\psi_i}$ we also abbreviate $\bX_{\psi_i}^{-1} \coloneqq \bX_{\psi_i^{-1}}$.

\subsection{Properties of triangle centers}

While the definition of triangle centers is broad and captures many desirable features, it does not impose strong constraints on triangle center functions. For instance, not all such functions yield triangle centers that have a well-defined representative in the Euclidean plane for all triangles, and some centers may coincide with others for multiple triangle configurations. In this section, we introduce terminology to rule out these pathologies.

\subsubsection{Traceable triangle centers}

While triangle center functions typically describe homogeneous coordinates, it is sometimes advantageous to work with triangle center functions in a normalized form. %To this end, we introduce:

\begin{definition}
    \label{def:NormalizationOfTriangleCenterFunction}
    Given a triangle center function \(\psi\in\cT\), its  \emph{trace} is given by 
    $$
    \nmlz{\psi}(a,b,c) = a\, \psi(a,b,c) +  b\, \psi(b,c,a) +  c\, \psi(c,a,b). 
    $$
    We say that the triangle center function \(\psi\in\cT\) is \emph{traceable} if its trace is nowhere vanishing, \ie, \(\nmlz{\psi}(a,b,c)\neq 0\) for all \((a,b,c)\in\triangle\),  and \emph{normalized} if $\nmlz{\psi}(a,b,c) \equiv 1$. Moreover, a triangle center is said to be traceable whenever its defining triangle center function is.  
\end{definition}

It follows from the definition of triangle centers that traceable triangle centers are those that correspond to well-defined points in \(\RR^2\) for all $\Delta \in \triangle$.  %\OG{non-degeneracy is assumed in the definition of \(\triangle\)}. 
Examples of such centers are listed\footnote{The computations used to determine traceability and non-traceability are provided in \appref{sec:app:section2}, \exref{ex:traceableTriangleCenters} and \exref{ex:nottraceableTriangleCenters}.} in \tblref{tab:trianglecenters}. Notably, the Feuerbach point $\bX_{11}$ is not traceable, since $\nmlz{\phi_{11}}(a,a,a) = 0$.

Finally, we note that the well-definedness of \defref{def:NormalizationOfTriangleCenterFunction} follows from the properties outlined in the following proposition, which are direct consequences of the properties of triangle center functions. 

\begin{proposition}
    \label{lem:propNormalization}
    Let $\psi \in \cT$ be a triangle center function of degree \(k\geq 0\) and $f\in\cycT$. Then, 
    \begin{enumerate}
         \item \(
         \nmlz{\psi}(a,b,c) = \nmlz{\psi}(b,c,a) = \nmlz{\psi}(c,a,b)\),
         \item \(\nmlz{\psi}(a,b,c) = \nmlz{\psi}(a,c,b)\),
         \item \(\nmlz{\psi}(ta,tb,tc) = t^{k+1} \,\nmlz{\psi}(a,b,c)\),
         \item \(\nmlz{f \psi}(a,b,c) = f(a,b,c)\, \nmlz{\psi}(a,b,c).\)
     \end{enumerate}
That is, $\nmlz{\psi} \in \cT$ is well defined on the equivalence classes of triangle center functions which describe the same triangle centers and homogeneous of  $\deg(\nmlz{\psi}) = k+1$.
\end{proposition}

Note that, if $\psi$ is traceable, $\nmlz{\psi} \in \cycT$. That is, any traceable triangle center function \(\psi\in\cT\) can be normalized by 
\begin{equation*}
    \psi\mapsto\tfrac{1}{\nmlz{\psi}}\psi.
\end{equation*}

With the above properties, the following Proposition is easy to verify. 
\begin{proposition}
    Normalized triangle center functions have homogeneity degree -1.
\end{proposition}

\subsubsection{Essentially different triangle centers}\label{sec:essentiallDifferentTriangleCenters}
For an equilateral triangle with side lengths \(a>0\), the trace of any triangle center function \(\psi\in\cT\) simplifies to 
$\nmlz{\psi}(a,a,a) = 3 a \psi(a,a,a)$. If $\psi(a,a,a) \neq 0$, which is the case for traceable triangle centers, the triangle center has barycentric coordinates
\[\bX_\psi = \begin{bmatrix}
    a\,\psi(a,a,a)\\ a\,\psi(a,a,a)\\ a\,\psi(a,a,a)
\end{bmatrix} \cong \begin{bmatrix}1\\ 1\\ 1\end{bmatrix} = \bX_2.\]
So, all traceable triangle centers coincide with the centroid of the triangle in equilateral triangles. Since this coincidence constitutes a degenerate scenario, we will mainly focus our attention on triangle centers that do not coincide, or are \emph{different}, for all other triangles.

\begin{definition}
    We call two triangle center functions $\psi_0, \psi_1\in\cT$ \emph{essentially different}, if the corresponding triangle centers coincide only in equilateral triangles.
\end{definition}

For example, in a right-angled isosceles triangle, the triangle centers $\bX_3$ and $\bX_{63}$ coincide at the midpoint of the hypotenuse. Consequently, the triangle centers are therefore not essentially different (see~\secref{sec:app:section2}, \exref{ex:essentiallyDifferent}). The following lemma provides pairs of essentially different triangle centers that are considered in later sections of this paper. 

\begin{lemma}\label{lem:essentiallDifferent}~
    \begin{enumerate}
        \item The triangle centers $\bX_3$ and $\bX_6$ 
        are essentially different.
        \item The triangle centers $\bX_{13}$, $\bX_{14}$, $\bX_{15}$, and $\bX_{16}$ are essentially different from the centroid $\bX_2$.
    \end{enumerate}
\end{lemma}

\subsection{Collinear triangle centers}\label{sec:collinear}

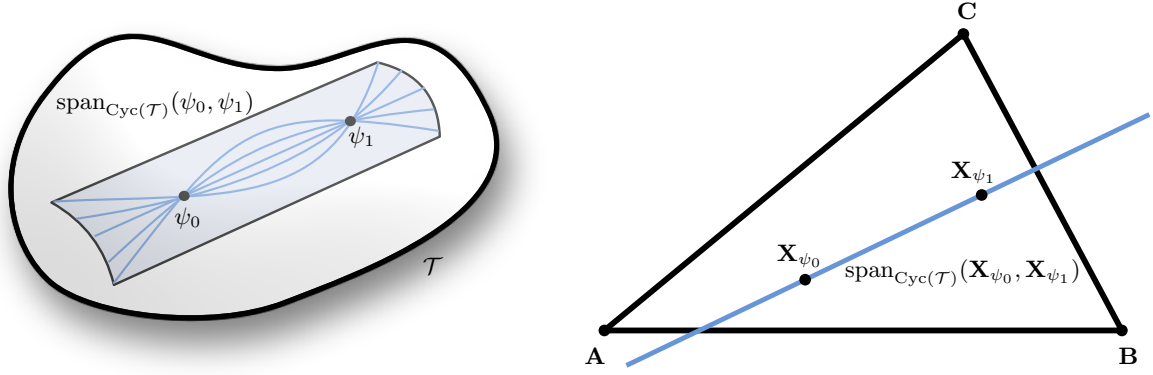
\begin{figure}
    \begin{footnotesize}    
    \begin{overpic}[width=\textwidth]{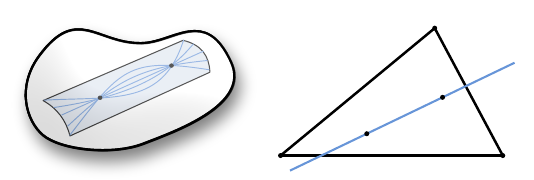}
    \put(38,9){$\cT$}
    \put(18,13.){$\psi_0$}
    \put(32,19.5){$\psi_1$}
    \put(8.5,22.25){$\spann_{\cycT}(\psi_0,\psi_1)$}
    \put(51,1.75){$\bA$}
    \put(94,1.75){$\bB$}
    \put(81,29.5){$\bC$}
    \put(66.5,10){$\bX_{\psi_0}$}
    \put(80.5,16.75){$\bX_{\psi_1}$}
    \put(72,8.75){$\spann_{\cycT}(\bX_{\psi_0}, \bX_{\psi_1})$}%
    \end{overpic}%
    \end{footnotesize}%
   \caption{Illustration of the space of triangle functions. }\label{fig:spacecenters}
\end{figure}

The study of triangle centers also examines their relationships, such as collinearity.
 Recall that three triangle centers $\bX_{\psi_0},\bX_{\psi_1},\bX_{\psi_2}$ are collinear, if for all $(a,b,c)\in \triangle$
 \begin{equation}\label{eqn:colX}
 \det\begin{pmatrix}
     \bX_{\psi_0} & \bX_{\psi_1} & \bX_{\psi_2} \\
     1 & 1 & 1
 \end{pmatrix} = 0.
 \end{equation}
Consequently, for triangles $\Delta\in\triangle$ for which $\nmlz{\psi_0}, \nmlz{\psi_1}, \nmlz{\psi_2}\neq 0$,  we define
$$
\bM_{\Delta} = \begin{pmatrix}
    \bA &
    \bB & \bC \\
    1 & 1 & 1
\end{pmatrix}
\quad \text{and} \quad 
\bM_{\psi} = \begin{pmatrix}
    \nicefrac{a \,\psi_0(a,b,c)}{\nmlz{\psi_0}} & \nicefrac{a\, \psi_1(a,b,c)}{\nmlz{\psi_1}} & \nicefrac{a \,\psi_2(a,b,c)}{\nmlz{\psi_2}} \\
    \nicefrac{b \,\psi_0(b,c,a)}{\nmlz{\psi_0}} & \nicefrac{b\, \psi_1(b,c,a)}{\nmlz{\psi_1}} & \nicefrac{b \,\psi_2(b,c,a)}{\nmlz{\psi_2}} \\
    \nicefrac{c \,\psi_0(c,a,b)}{\nmlz{\psi_0}} & \nicefrac{ c \,\psi_1(c,a,b)}{\nmlz{\psi_1}}& \nicefrac{c \,\psi_2(c,a,b)}{\nmlz{\psi_2}} 
\end{pmatrix}, 
$$
and rewrite 
$$
\det\begin{pmatrix}
     \bX_{\psi_0} & \bX_{\psi_1} & \bX_{\psi_2} \\
     1 & 1 & 1
 \end{pmatrix} =
\det \left(\bM_{\Delta} \cdot \bM_{\psi}\right) = \det \left(\bM_{\Delta} \right) \, \det \left(\bM_{\psi}\right).
$$
Since the vertices of a triangle \(\Delta\in\triangle\) are in general position, $\det(\bM_{\triangle}) \neq 0$, Equation~\eqref{eqn:colX} is equivalent to $\det(\bM_{\psi}) = 0$. Then, due to the multilinearity of the determinant, we may simplify 
 $$
 \det(\bM_{\psi})= \frac{a b c}{\nmlz{\psi_0}\nmlz{\psi_1}\nmlz{\psi_2}}
 \det \begin{pmatrix}
     \psi_0(a,b,c) &  \psi_1(a,b,c) & \psi_2(a,b,c) \\
     \psi_0(b,c,a) &  \psi_1(b,c,a) & \psi_2(b,c,a) \\
     \psi_0(c,a,b) & \psi_1(c,a,b) & \psi_2(c,a,b) 
\end{pmatrix}.
 $$

As we focus on non-degenerate triangles, a necessary and sufficient condition for the collinearity of three triangle centers is the linear dependence of triples corresponding to triangle center functions evaluated at cyclic permutation triangle edge lengths.

 \subsubsection{Cyclic-affine triangle center combinations}\label{sec:CyclicAffineTriangleCenterCombinations}
We now formalize a simple but flexible notion of linear dependence among triangle center functions, where the coefficients are allowed to vary cyclically with $(a,b,c)$.

\begin{definition}\label{lem:cyclicaffinecombination}
Let $\psi_0, \psi_1 \in \cT$ be two triangle center functions and $\omega_0,\omega_1\in\cycT$  such that $\deg(\omega_0\psi_0) = \deg(\omega_1\psi_1)$. We refer to the triangle center function 
\begin{equation}\label{eqn:cyclicaffine}
\psi_{\bomega}(a,b,c) = \omega_0(a,b,c)\, \psi_0(a,b,c) + \omega_1(a,b,c)\, \psi_1(a,b,c),
\end{equation}
as a \emph{cyclic-affine combination of} $\psi_0$ and $\psi_1$.
We denote the set of cyclic-affine combinations of $\psi_0$ and $\psi_1$ by $\spann_{\cycT}(\psi_0,\psi_1)$.
\end{definition}

It is straightforward to verify that $\psi_{\bomega}$ in \teqref{eqn:cyclicaffine} is a triangle center function, and therefore is well defined.

\begin{remark}
Note that even if the two triangle center functions  in \lemref{lem:cyclicaffinecombination} are traceable, the corresponding cyclic-affine combination need not be. In particular, for normalized triangle center functions $\psi_0$ and $\psi_1$, $\omega_0\in\cycT$, and $\omega_1 = -\omega_0$, we have $\Sigma_{\psi_{\bomega}}(a,b,c) = 0$ for all $\Delta \in\triangle$.
\end{remark}

The next lemma shows that cyclic-affine combinations behave like genuine affine combinations at the level of points: whenever the combination is traceable, the resulting center lies on the line through $\bX_{\psi_0}$ and $\bX_{\psi_1}$, with weights determined by the cyclic coefficients, see \figref{fig:spacecenters}.

\begin{lemma}\label{lem:cycaffineCenter}
Let $\psi_{\bomega} \in \spann_{\cycT}(\psi_0,\psi_1)$ be traceable. Then, the associated triangle centers are collinear and given by 
\begin{equation}\label{eqn:x_omega}
\bX_{\psi_{\bomega}} = \frac{\omega_0\nmlz{\psi_0}}{\omega_0\nmlz{\psi_0}+ \omega_1\nmlz{\psi_1}} \bX_{\psi_0} +  \frac{\omega_1 \nmlz{\psi_1}}{\omega_0\nmlz{\psi_0}+ \omega_1\nmlz{\psi_1}} \bX_{\psi_1}.
\end{equation}
In addition, we have that
 $$
    \psi_{\bomega} \in \spann_{\cycT}(\psi_0, \psi_1) \quad \Longrightarrow \quad 
    \psi_{1} \in \spann_{\cycT}(\psi_0, \psi_{\bomega}).
    $$ 
 \end{lemma}

 Note, however, that somewhat counterintuitively $\psi_0, \psi_1 \not\in \spann_{\cycT}(\psi_0, \psi_1)$ since, by \defref{def:cycT}, $\omega_0, \omega_1 \not\equiv 0$.
Consequently, $\spann_{\cycT}(\psi_0, \psi_1)  \neq \spann_{\cycT}(\psi_0, \psi_{\bomega})$.

A key point is that this construction depends only on the underlying triangle centers, not on how we choose to represent them. In particular, rescaling $\psi_0$ and $\psi_1$ by cyclic factors does not change the family of cyclic-affine combinations they generate.
 \begin{lemma}\label{lem:cycAffineIndependence}
The set of cyclic-affine combinations of two triangle centers functions $\psi_0$ and $\psi_1$ is independent of the choice of their representatives. That is, 
$$
\spann_{\cycT}(f_0 \psi_0, f_1 \psi_1) = \spann_{\cycT}(\psi_0,\psi_1),
$$
where $f_0,f_1\in\cycT$.
Consequently, we may also refer to the set of triangle centers $\bX_{\psi_{\bomega}}$ corresponding to cyclic-affine triangle center functions of $\psi_0$ and $\psi_1$ as $\spann_{\cycT}(\bX_{\psi_0}, \bX_{\psi_1})$. 
\end{lemma}

 It is known that the line connecting $\bX_3$ and $\bX_6$, the so-called \emph{Brocard axis}, contains the first isodynamic center $\bX_{15}$. \exref{ex:X15inBrocard} in \appref{sec:app:section2} illustrates $\bX_{15}$ as a cyclic-affine combination of $\bX_3$ and $\bX_6$.

Finally, we discuss under which conditions $\spann_{\cycT}(\bX_{\psi_0}, \bX_{\psi_1})$ contains all  triangle centers that are essentially different from $\bX_{\psi_0}$ and $\bX_{\psi_1}$.

\begin{lemma}\label{lem:triangleLineCycAffine}
If two distinct triangle centers $\bX_{\psi_0}$ and $\bX_{\psi_1}$ are both traceable, then $\spann_{\cycT}(\bX_{\psi_0}, \bX_{\psi_1})$ contains all traceable triangle centers that are collinear with $\bX_{\psi_0}$ and $\bX_{\psi_1}$ and are essentially different from these two centers.    
\end{lemma}

\subsubsection{Constant-affine triangle center combinations}\label{sec:constantAffineCombos}

Next, we highlight a subset of collinear triangle centers whose coefficients do not depend on the shape of the triangle, \ie, are constant with respect to the triangle edge lengths $a$, $b$, and $c$.

\begin{definition}\label{lem:combinationsConstantAffine}
Let $\psi_0$ and $\psi_1$ be two triangle center functions with $\deg(\psi_0) = \deg(\psi_1)$ and $\gamma_0, \gamma_1 \in \RR\setminus\{0\}$. We refer to the triangle center function
$$
\psi_{\lambda_0 : \lambda_1}(a,b,c) = \lambda_0\, \psi_0(a,b,c) + \lambda_1 \,\psi_1(a,b,c),
$$
as a \emph{constant-affine combination} of $\psi_0$ and $\psi_1$. We denote the set of constant-affine combinations of $\psi_0$ and $\psi_1$ by $\spann_{\const}(\psi_0, \psi_1)$.
\end{definition}

Since constant-affine combinations of triangle centers are a special case of cyclic-affine combinations, a special case of  \lemref{lem:cycaffineCenter} applies verbatim in this setting.

\begin{proposition}
    \label{lem:constaffineCenter}
Let $\psi_{\lambda_0:\lambda_1} \in \spann_{\const}(\psi_0, \psi_1)$ be traceable. Then, the associated triangle centers are collinear and given by
\begin{equation}\label{eqn:psiomega}
\bX_{\psi_{\lambda_0:\lambda_1}} = \frac{\lambda_0\nmlz{\psi_0}}{\lambda_0\nmlz{\psi_0} + \lambda_1 \nmlz{\psi_1}} \bX_{\psi_0} + \frac{\lambda_1\nmlz{\psi_1}}{\lambda_0\nmlz{\psi_0} + \lambda_1 \nmlz{\psi_1}} \bX_{\psi_1}. 
\end{equation}
In addition, we have that 
 $$
    \psi_{\lambda_0:\lambda_1} \in \spann_{\const}(\psi_0, \psi_1) \quad \Longrightarrow \quad 
    \psi_{1} \in \spann_{\const}(\psi_0, \psi_{\lambda_0:\lambda_1}).
    $$
\end{proposition}

Note that if the two triangle center functions $\psi_0$ and $\psi_1$ in \lemref{lem:constaffineCenter} are normalized and therefore have homogeneity degree $-1$, both the coefficients of the triangle center function and the triangle centers are constant and the resulting triangle center function $\psi_{\lambda_0:\lambda_1}$ is normalized. If, in addition, $\lambda_0 + \lambda_1 = 1$, the coefficients of the triangle center function and the triangle center are the same.

Notably, constant-affine combinations produce families of triangle centers that are pairwise essentially different.

\begin{lemma}\label{lem:essDifferentConstAffine}
If the triangle center functions $\psi_0$ and $\psi_1$ in \lemref{lem:constaffineCenter} correspond to essentially different traceable triangle centers, then for $\lambda_0, \lambda_1\neq 0$, the triangle centers corresponding to $\psi_{\lambda_0:\lambda_1}$ are essentially different from $\psi_0$ and $\psi_1$. Additionally, if $\nicefrac{\lambda_0}{\lambda_1} \neq \nicefrac{\overline{\lambda}_0}{\overline{\lambda}_1}$, then the triangle center functions $\psi_{\lambda_0:\lambda_1}$ and $\psi_{\overline{\lambda}_0:\overline{\lambda}_1}$ are essentially different. 
\end{lemma}

Finally, note that, unlike cyclic-affine combinations, constant-affine combinations depend on the choice of triangle center functions. 
Nevertheless, using three collinear triangle centers, we can define a set of constant-affine combinations that are independent of the specific choice of triangle center function. For more information, see~\appref{sec:app:section2}, \lemref{lem:AffineFromCyclic}. An example of families of constant-affine triangle centers within the Brocard axis is provided in \exref{ex:brocard}.

\section{$\bPsi$-Triangle families}\label{sec:PsiTriangleFamilies}
In this section, we explore families of triangles that can be derived from a given triangle \(\Delta\in \triangle\). Inspired by Kimberling's approach, which employs triangle center functions as abstract tools for studying triangle centers, we describe these triangle families in terms of triplets of functions.

\begin{definition}
    \label{def:TriangleFamilyFunction}
    A \emph{\(\bPsi\)-triangle family} \(t\mapsto \Delta_t^{\bPsi}\in\triangle\) is determined by a triangle \(\Delta\in \triangle\) and three scalar valued functions $\Psi_i(t)$, 
    \begin{equation*}
\bPsi\colon\RR\to \RR^3,\ t\mapsto \begin{pmatrix}
    \Psi_1(t) \\ \Psi_2(t) \\ \Psi_3(t)
\end{pmatrix},
    \end{equation*} 
    such that \(\Psi_1(t) + \Psi_2(t) + \Psi_3(t)\neq 0\) and not \(\Psi_1(t) \equiv \Psi_2(t) \equiv \Psi_3(t)\). In terms of the barycentric coordinates with respect the vertices \(\bA,\bB\), and \(\bC\) of \(\Delta\), the vertices of \(\Delta_t^{\bPhi}\) are given by
    \begin{equation*}
        \bA_t^{\bPsi} = \begin{bmatrix}\Psi_1(t) \\ \Psi_2(t) \\ \Psi_3(t)\end{bmatrix}, \quad 
        \bB_t^{\bPsi} = \begin{bmatrix}\Psi_3(t) \\ \Psi_1(t) \\ \Psi_2(t)\end{bmatrix}, \quad
        \bC_t^{\bPsi} = \begin{bmatrix}\Psi_2(t) \\ \Psi_3(t) \\ \Psi_1(t)\end{bmatrix}.
    \end{equation*}
    In the following, we refer to the three functions \(\Psi_1(t)\), \(\Psi_2(t)\), and \(\Psi_3(t)\) as the \emph{generating functions} of the family \(\Delta_t^{\bPsi}\). Additionally, we call a triangle family \emph{interpolating}, if $\Delta_t^{\bPsi} = \Delta$ for some $t$. 
\end{definition}

Similar to triangle center functions, we follow the convention that generic triangle families are denoted by $\bPsi$, and specific ones are represented by $\bPhi$ with a subscript. For example,  
\begin{equation}\label{eqn:identity}
\bPhiID \colon t \mapsto \begin{pmatrix}
    1\\0\\0
\end{pmatrix}
\end{equation}
is the identity mapping, and 
\begin{equation}\label{eqn:scaling}
 \bPhiS \colon t \mapsto 
 \begin{pmatrix}1 + 2t\\ 1 -t\\ 1-t\end{pmatrix}
\end{equation}
 represents a homothety with the centroid as the origin. 

The following lemma identifies two invariants of $\bPsi$-triangle families that can serve as necessary conditions for a triangle family to be classified as a $\bPsi$-triangle family:

\begin{lemma}\label{lem:propertiesFamilies}
 For all $t\in\RR$, the following statements hold:
\begin{enumerate}
    \item The centroid of a triangle \(\Delta^\bPsi_t\) coincides with the centroid of \(\Delta\).
    \item If $\Delta$ is equilateral, $\Delta_t^{\bPsi}$ is also equilateral. 
    \end{enumerate}
\end{lemma}

This lemma provides a useful criterion for determining when a given triangle family is not a $\bPsi$-triangle family. For example, families of triangles obtained by linear displacements of one of the vertices, cannot be examples of $\bPsi$-triangle families as they do not preserve the centroid~\cite{Jurkin:2022:LOC, Kodrnja:2023:LCT}. Moreover, the centroids of Poristic triangles generically trace a circle~\cite[Thm. 4.6]{Odehnal:2011:PLT}.

\subsection{Concatenation of $\bPsi$-triangle families}
Having established some basic properties of $\bPsi$-triangle families, it is natural to consider the concatenation of triples of generating functions corresponding to two triangle families:

\begin{deflemma}\label{deflemma:conctrianglefamilies}
The concatenation of two triangle families, 
\begin{equation*}
    \bPsi = \begin{pmatrix}
    \Psi_1 \\ \Psi_2 \\ \Psi_3
    \end{pmatrix}, \qquad \tilde\bPsi = \begin{pmatrix}
    \tilde\Psi_1 \\ \tilde\Psi_2 \\ \tilde\Psi_3
    \end{pmatrix},
\end{equation*}
is a $\bPsi$-triangle family given by 
\begin{equation}\label{eqn:concatenationFactor}
\bPsi \circ \tilde\bPsi\colon t \mapsto\begin{pmatrix}
    \Psi_1\tilde\Psi_1 + \Psi_2\tilde\Psi_3 + \Psi_3\tilde\Psi_2 \\ 
    \Psi_1\tilde\Psi_2 + \Psi_2\tilde\Psi_1 + \Psi_3\tilde\Psi_3 \\ 
    \Psi_1\tilde\Psi_3 + \Psi_2\tilde\Psi_2 + \Psi_3\tilde\Psi_1
\end{pmatrix}.
\end{equation}
\end{deflemma}

Given that the set of all $\bPsi$-triangle families is closed under concatenation, the following lemma lists additional algebraic properties of this structure:

\begin{lemma}\label{lem:proptriangleConcatenation} 
The following properties hold for all 
$\bPsi$-triangle families: 
\begin{enumerate}
    \item The concatenation of triangle families commutes: 
$$
\bPsi \circ \tilde\bPsi = \tilde\bPsi \circ \bPsi.
$$
 \item The concatenation of triangle families is associative: 
 $$
 \bPsi \circ \left( \tilde\bPsi \circ \hat{\bPsi}\right) = \left(\bPsi \circ \tilde\bPsi\right) \circ \hat{\bPsi}.
 $$
 \item $\bPhiID$ is a neutral element with respect to the concatenation, 
 $$
 \bPsi \circ \bPhiID = \bPsi.
 $$
 \item Every $\bPsi$-triangle family has an inverse, 
 \begin{align*}
 \bPsi^{-1}\colon t \mapsto %\left( 
 \begin{pmatrix}
  (\Psi_1)^2 - \Psi_2\Psi_3\\
  (\Psi_3)^2 - \Psi_1 \Psi_2\\
 (\Psi_2)^2 - \Psi_3 \Psi_1 
 \end{pmatrix}.
\end{align*}
\end{enumerate}
Consequently, the set of all $\bPsi$-triangle families in combination with the concatenation is an Abelian group with neutral element $\bPhiID$.
\end{lemma}

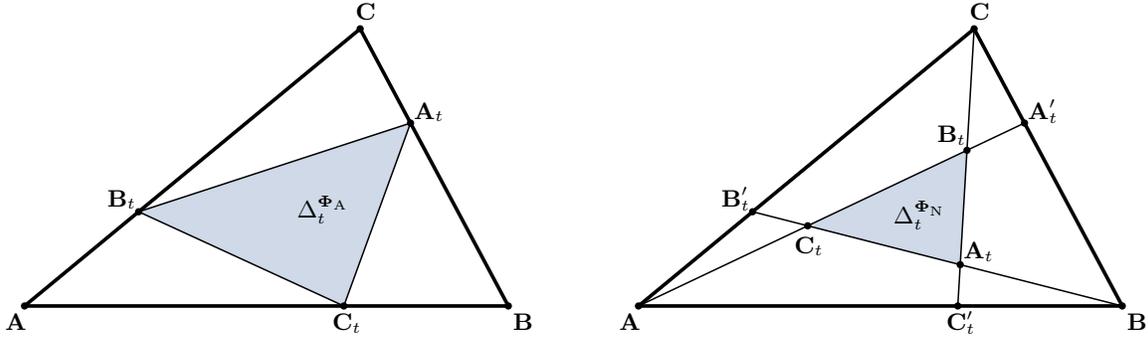
\begin{figure}[t]
    \begin{footnotesize}
    \begin{overpic}[width=0.47\textwidth]{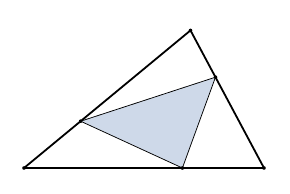}
          \put(5,4.75){$\bA$}
    \put(92,4.75){$\bB$}
    \put(65,58){$\bC$}
    \put(61,4.75){$\bC_t$}
    \put(75,41){$\bA_t$}
    \put(22.5,26){$\bB_t$}
    \put(55,24){$\Delta_t^{\bPhiA}$}
    \end{overpic}
    \hfil
    \begin{overpic}[width=0.47\textwidth]{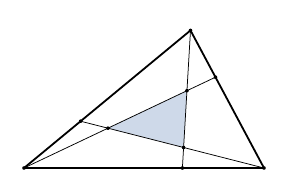}
        \put(5,4.75){$\bA$}
    \put(92,4.75){$\bB$}
    \put(65,58){$\bC$}
    \put(61,4.75){$\bC_t'$}
    \put(75,41){$\bA_t'$}
    \put(22.5,26){$\bB_t'$}
    \put(64,17){$\bA_t$}
    \put(59.5,37){$\bB_t$}
    \put(35,18){$\bC_t$}
     \put(52,23){$\Delta_t^{\bPhiN}$}
    \end{overpic}
     \end{footnotesize}%
\caption{Illustration of aliquot (left) and nedian (right) triangle families.}\label{fig:aliquotNedianConstruction}
\end{figure}

\subsection{Aliquot and nedian triangle families}
 In this paper, two triangle families, the \emph{aliquot} and the \emph{nedian} families associated with a triangle~\cite{Satterly:1956:NNT}, play an important role. The former triangle family is obtained by equipping the triangle \(\Delta\in\triangle\) with vertices \(\bA, \bB\) and \(\bC\) with an orientation and defining a new triple of vertices \(\bA_t, \bB_t, \bC_t\) such that the placement of the new vertices on the edges divides the respective edge in equal proportions (\figref{fig:aliquotNedianConstruction}). This can be achieved by taking linear combinations 
\[\bA_t = (1-t) \bB + t \bC, \quad \bB_t = (1-t) \bC + t \bA, \quad \bC_t = (1-t) \bA + t \bB.\]
This construction can be stated in terms of a \(\bPsi\)-triangle family as follows:

\begin{definition}
   The \emph{aliquot family} $\bPhiA$ is  defined by 
        \begin{equation}\label{eqn:aliquot}
        \bPhiA \colon t \mapsto \begin{pmatrix}
            0\\ 1-t\\ t
        \end{pmatrix}.
    \end{equation}
\end{definition}

Instead of choosing points \(\bA_t', \bB_t'\), and \(\bC_t'\) that determine the aliquot family as a new triangle, we can choose \(\bA_t, \bB_t\), and \(\bC_t\) to be respective pairwise intersections %\(\bA, \bB\) and \(\bC\) 
 of the lines \(\ell_{\bA\bA_t'}, \ell_{\bB\bB_t'}\), and \(\ell_{\bC\bC_t'}\), see \figref{fig:aliquotNedianConstruction}. This determines yet another \(\bPsi\)-triangle family, the \emph{nedian family} $\bPhiN$, which can be derived from an initial triangle as stated in the next lemma. 
\begin{deflemma}\label{lem:nedian}
    The \emph{nedian family} $\bPhiN$ is defined by 
    \begin{equation}\label{eqn:nedian}
    \bPhiN \colon t \mapsto \begin{pmatrix}
        (1-t)t\\ t^2\\ (1-t)^2
    \end{pmatrix} .
    \end{equation}
\end{deflemma}

The following proposition collects basic geometric properties of the aliquot and nedian families, all of which follow directly from the definitions and are illustrated in \figref{fig:triangleCurvesFromAliquotNedian}.

\begin{proposition}\label{prop:families01}
Let $\Delta = (\bA, \bB, \bC)$ be a triangle in $\triangle$,  $\bX_2$ its centroid.
\begin{enumerate}
    \item The aliquot and nedian families of a triangle are interpolating. Specifically, 
    $$
    \Delta_0^{\bPhiA} = (\bB,\bC, \bA), \quad \Delta_1^{\bPhiA} = (\bC, \bA, \bB), \quad 
    \Delta_0^{\bPhiN} = (\bC, \bA, \bB), \quad
    \text{and}\quad
    \Delta_1^{\bPhiN} = (\bA, \bB, \bC).
    $$
    \item $\Delta_{\nicefrac{1}{2}}^{\bPhiA}$ is related to the original triangle by a scaling of $\nicefrac{1}{2}$ with respect to $\bX_2$ followed by a point inversion through  $\bX_2$.
    \item The triangles $\Delta_{\nicefrac{1}{3}}^{\bPhiA}$ and $\Delta_{\nicefrac{2}{3}}^{\bPhiA}$ relate by a point inversion through $\bX_2$.
    \item The triangle $\Delta_{\nicefrac{1}{2}}^{\bPhiN}$ degenerates to $\bX_2$. 
\end{enumerate}
\end{proposition}

\begin{figure}[t]
        \begin{footnotesize}
    \begin{overpic}[width=0.47\textwidth]{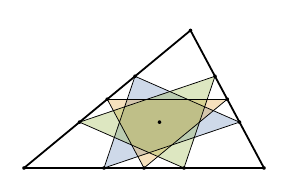}
          \put(5,4){$\bA$}
    \put(92,4){$\bB$}
    \put(65,58){$\bC$}
    \put(84,25){$\Delta_{\nicefrac{1}{3}}^{\bPhiA}$}
    \put(80,33){$\Delta_{\nicefrac{1}{2}}^{\bPhiA}$}
    \put(76,41){$\Delta_{\nicefrac{2}{3}}^{\bPhiA}$}
     \put(56.5,23){$\bX_2$}
    \end{overpic}
    \hfil
    \begin{overpic}[width=0.47\textwidth]{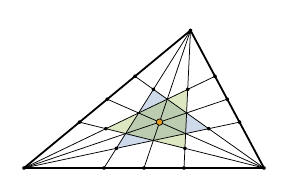}
        \put(5,4){$\bA$}
    \put(92,4){$\bB$}
    \put(65,58){$\bC$}
    \put(41.5,12.5){$\Delta_{\nicefrac{1}{3}}^{\bPhiN}$}
    \put(53.5,12.5){$\Delta_{\nicefrac{2}{3}}^{\bPhiN}$}
    \put(56,26){$\Delta_{\nicefrac{1}{2}}^{\bPhiN}$}
    \end{overpic}%
     \end{footnotesize}%
    \caption{Illustrations of properties of aliquot and nedian triangle families discussed in \lemref{prop:families01}.}\label{fig:triangleCurvesFromAliquotNedian}
\end{figure}
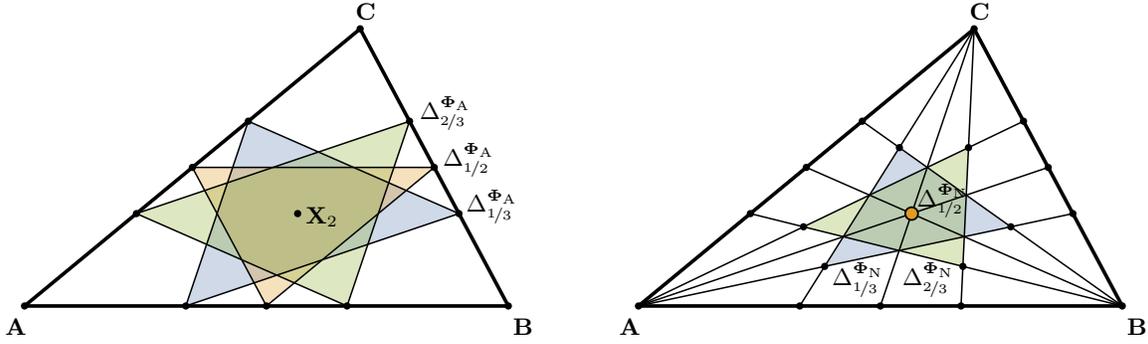

In addition to these results, we show a connection between the $\bPsi$-triangle families introduced so far and their inverses. Although these observations have no direct implications for the theory developed in the next section, we find them nonetheless noteworthy. 

\begin{lemma}\label{lem:inversesbPhi}
The $\bPsi$-triangle families $\bPhiS$, $\bPhiA$, and $\bPhiN$ relate to their inverses as follows:
\begin{itemize}
    \item For $t\neq 0$, the inverse of the family of scaled triangles in \teqref{eqn:scaling} is a scaling by $\frac{1}{t}$,
    $$
    \bPhiS^{-1}(t) = \bPhiS\left(\frac{1}{t}\right).
    $$ 
    \item For $t\neq \nicefrac{1}{2}$, the inverse of the aliquot family in \teqref{eqn:aliquot} is a member of the nedian family, 
    $$
    \bPhiA^{-1}(t) = \bPhiN\left(-\frac{t}{1-2t}\right).
    $$
    \item  For  \(t\neq\nicefrac{1}{2}\), the inverse of the nedian family in \teqref{eqn:nedian} is a member of the aliquot family,
    $$
    \bPhiNinv(t) = \bPhiA\left(-\frac{t}{1-2t}\right).
    $$
\end{itemize}
\end{lemma}
Finally, we note that by concatenation of the of the aliquot and nedian families we can obtain a wealth of non-trivial examples of $\bPsi$-families.

\subsection{Relating triangle families to the aliquot family}

Finally, we show which $\bPsi$-triangle families relate to the aliquot family through appropriate scaling and reparametrization. This observation will play an important role at several points in the next section. Before we continue, we first introduce the following subset of $\bPsi$-triangle families: 

\begin{definition}
Given a $\bPsi$-triangle family, define  \begin{equation*}
    \delta_{\bPsi}(t) = 2\Psi_1(t) - \Psi_2(t)-\Psi_3(t)\in\RR.
\end{equation*} We refer to a $\bPsi$-triangle family as \emph{decomposable}, if for all $t\in\RR$, the condition $\delta_{\bPsi}(t) = 0$ implies $\Psi_1(t) = \Psi_2(t) = \Psi_3(t)$.    
\end{definition}

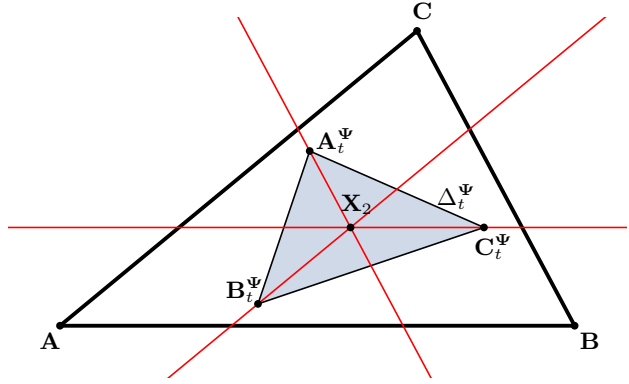
\begin{figure}[t]
    \centering
    \begin{footnotesize}
    \begin{overpic}[width=0.5\textwidth]
    {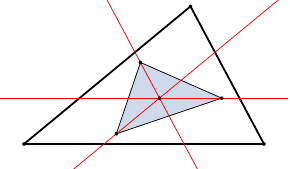}
       \put(5,4.75){$\bA$}
    \put(92,4.75){$\bB$}
    \put(65,58){$\bC$}
    \put(53.5,27){$\bX_2$}
    \put(75,20){$\bC_t^{\bPsi}$}
    \put(35,13){$\bB_t^{\bPsi}$}
    \put(49.5,37){$\bA_t^{\bPsi}$}
    \put(69,28){$\Delta_t^{\bPsi}$}
    \end{overpic}%
        \end{footnotesize}%
    \caption{Illustration of a triangle in a not decomposable triangle family $\bPsi$.}\label{fig:NotDecomposableTriangleFamily}
\end{figure}

Geometrically, a decomposable $\bPsi$-triangle family does not contain a triangle $\Delta_t^{\bPsi}$ whose vertices are positioned in the following configuration (see \figref{fig:NotDecomposableTriangleFamily}). For a triangle $\Delta\in\triangle$, let  $\ell_{\bA\bB}$, $\ell_{\bB\bC}$, and $\ell_{\bC\bA}$ denote the three lines that pass through the centroid $\bX_2$ and are parallel to $\bA\bB$, $\bB\bC$, and $\bC\bA$, respectively. The triangle family is called decomposable if, for all $t\in\RR$, $\bA_t^{\bPsi}$, $\bB_t^{\bPsi}$, $\bC_t^{\bPsi}$ lying on the lines $\ell_{\bB\bC}$, $\ell_{\bC\bA}$, and $\ell_{\bA\bB}$, respectively, implies that $\bA_t^{\bPsi} = \bB_t^{\bPsi} = \bC_t^{\bPsi} = \bX_2$.

Straightforward computations verify that non-trivial examples of decomposable families include the aliquot family $\bPhiA$ and the nedian family $\bPhiN$. On the other hand, an example of a not decomposable family is 
$$
\bPsi \colon t \mapsto \begin{pmatrix}
    1-t\\ t\\ 0
\end{pmatrix},$$
since $\delta_{\bPsi}(2) = 0$, but $\bPsi(2) = (-1,2,0)$.

The next lemma justifies the term \emph{decomposable} by demonstrating that each triangle family in this class can be expressed as a composition of simpler, well-understood transformations, making the underlying structure explicit.
\begin{lemma}\label{lem:decompositionSigmaTau}
    A decomposable $\bPsi$-triangle family allows a decomposition  
$$
\bPsi(t) = \bPhiS(\sigma(t))\circ \bPhiA(\tau(t)).
$$
where 
\begin{equation}\label{eqn:sigmatau}
\sigma(t) = -\frac{2\Psi_1(t) - \Psi_2(t) - \Psi_3(t)}{\Psi_1(t) + \Psi_2(t) + \Psi_3(t)}
\quad \text{and} \quad 
    \tau(t) = \frac{\Psi_1(t) - \Psi_3(t)}{2 \Psi_1(t) -  \Psi_2(t) - \Psi_3(t)}.
\end{equation}
If $\Psi_1(t) = \Psi_2(t) = \Psi_3(t)$, then $\sigma(t) = 0$ while $\tau(t)$ is undefined.
\end{lemma}

\begin{proof} 
To arrive at constraints for $\sigma(t)$ and $\tau(t)$, we compare a normalized  form of $\bPsi(t)$ and $\bPhiS(\sigma(t))\circ \bPhiA(\tau(t))$. Specifically, by using the normalized version of the expression in \teqref{eqn:concatenationFactor} (or directly \teqref{eqn:concatenation}), we obtain
\begin{align}\label{eqn:PhiSigmaTau}
\frac{\Psi_1(t)}{\Psi_1(t) + \Psi_2(t) + \Psi_3(t)} &= \frac{1}{3}\left(1 - \sigma(t)\right) \nonumber\\
   \frac{\Psi_2(t)}{\Psi_1(t) + \Psi_2(t) + \Psi_3(t)} &= \frac{1}{3}\left(1 +\sigma(t) \left(2 - 3 \tau(t)\right)\right)\\
    \frac{\Psi_3(t)}{\Psi_1(t) + \Psi_2(t) + \Psi_3(t)} &= \frac{1}{3}\left( 1-\sigma(t) \left(1 - 3\tau(t)\right) \right).\nonumber
\end{align}
Solving the first equation for $\sigma(t)$ results in the stated expression for $\sigma(t)$. Inserting $\sigma(t)$ into the second and third equations simplifies to the same expression, namely,
$$
(2 \Psi_1(t) - \Psi_2(t) - \Psi_3(t)) \tau(t) = \Psi_1(t) - \Psi_3(t).
$$
Note that if $\Psi_1(t) = \Psi_2(t) = \Psi_3(t)$, then $\sigma(t) = 0$ and $\tau(t)$ is undefined. Otherwise, since $\bPsi$ is decomposable, we have that $2\Psi_1(t) - \Psi_2(t) - \Psi_3(t) \neq 0$ and obtain the stated expression for $\tau(t)$.
\end{proof}

\begin{corollary}\label{cor:nediantoaliquot}
For all $t\in\RR$, we have
$   \bPhiN(t)  = \bPhiS(\sigma(t)) \circ \bPhiA(\tau(t))$,
    where 
\begin{align}\label{eqn:limaconTauSigma}
\tau(t) = \frac{1-t}{1-2t} \quad \text{and}  \quad
\sigma(t) = \frac{(1 - 2t)^2}{1-t+t^2}.
\end{align}
\end{corollary}

\section{Triangle curves}\label{sec:trianglecurves}

With these preparations in place, we now turn our attention to curves that are the locus of a triangle center in $\bPsi$-triangle family: 

\begin{definition} We consider curves $t \mapsto \bX_{\psi}^\bPsi(t)$ traced by triangle centers $\bX_{\psi}$ in a $\bPsi(t)$-triangle family. We refer to these as \emph{triangle curves}. 
\end{definition}

\begin{remark}
    While we implicitly assume, as with triangle center functions, that triangle families are continuous and generally produce continuous curves, this is not  a prerequisite.
\end{remark}

\begin{remark}
Since $\bPsi$-triangle families of equilateral triangles consist of equilateral triangles (see \lemref{lem:propertiesFamilies}), and traceable triangle centers coincide with the centroid (see \secref{sec:essentiallDifferentTriangleCenters}), we exclude equilateral triangles from the following discussion.
\end{remark}

\begin{figure}[t]
\begin{subfigure}{\textwidth}
        \centering
        \includegraphics[width=0.32\textwidth, trim = 0.4cm 0cm 0.4cm 0cm, clip]{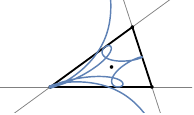}
         \includegraphics[width=0.32\textwidth , trim = 0.4cm 0cm 0.4cm 0cm, clip]{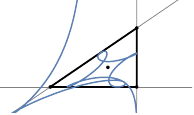}
          \includegraphics[width=0.32\textwidth, trim = 0.4cm 0cm 0.4cm 0cm, clip]{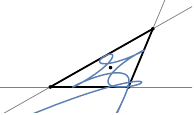}
        \caption{Trace of a generic triangle center ($\bX_{59}$). }
        \label{fig:triangle_curves_gen}
    \end{subfigure}
    
    \vspace{3mm}
    
    \begin{subfigure}{\textwidth}
        \centering
        \includegraphics[width=0.32\textwidth, trim = 0.4cm 0cm 0.4cm 0cm, clip]{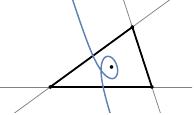}
         \includegraphics[width=0.32\textwidth, trim = 0.4cm 0cm 0.4cm 0cm, clip]{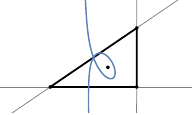}
          \includegraphics[width=0.32\textwidth, trim = 0.4cm 0cm 0.4cm 0cm, clip]{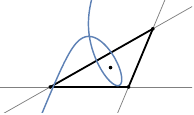}
        \caption{Trace of a semi-invariant curve-generating center ($\bX_3$). }
        \label{fig:triangle_curves_semi}
    \end{subfigure}

    \vspace{3mm}
    
    \begin{subfigure}{\textwidth}
        \centering
         \includegraphics[width=0.32\textwidth, trim = 0.4cm 0cm 0.4cm 0cm, clip]{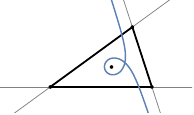}
         \includegraphics[width=0.32\textwidth, trim = 0.4cm 0cm 0.4cm 0cm, clip]{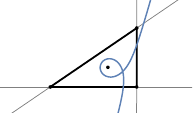}
          \includegraphics[width=0.32\textwidth, trim = 0.4cm 0cm 0.4cm 0cm, clip]{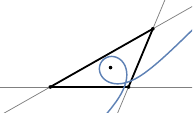}
        \caption{Trace of an invariant triangle curve-generating center ($\bX_{13}$). }
        \label{fig:triangle_curves_inv}
    \end{subfigure}
\caption{Illustration of triangle curves for different triangles of $\bPhiA$ triangle families. The centrally located points represent the triangle's centroids. }\label{fig:triangleCurvesIll}
\end{figure}

\subsection{Properties of triangle curves}\label{sec:propertiesOfTriangleCurves}

Given a triangle center and a triangle family, we distinguish between local and global properties of the resulting triangle curve. Local properties focus on specific configurations and selected parameter values, such as which triangle centers are interpolated by the triangle curve. On the other hand, global properties analyze the overall geometry of the triangle curve.

For some triangle families, it is relatively straightforward to determine the resulting triangle curves. For example, the triangle family of scaled triangles $\bPhiS$ in \teqref{eqn:scaling}, the curves $\bX_{\psi}^{\bPhiS}(t)$ parametrize lines connecting $\bX_{\psi}$ to $\bX_2$.
While some triangle curves exhibit visual similarities, they generally depend on the shape of the underlying triangle, see \figref{fig:triangle_curves_gen}.

\subsubsection{Local properties}
In the following section, we present local features of triangle curves, \ie, properties corresponding to specific parameter values,  beginning with a property that follows directly from their definition:

\begin{proposition}
If $\Delta_t^{\bPsi} = \Delta$, then $\bX_{\psi}^{\bPsi}(t) = \bX_{\psi}$.
\end{proposition} 

However, in general, local properties depend on the triangle family or triangle center. The following lemma summarizes features of triangle curves derived from the corresponding properties of the aliquot and nedian triangle families, as presented in \propref{prop:families01}.

\begin{proposition}\label{prop:propAliquotNedian}
Let $\bPhiA$ represent the aliquot triangle family described in Equation~\eqref{eqn:aliquot},  $\bPhiN$ the nedian triangle family in Equation~\eqref{eqn:nedian}, and let $\psi\in\cT$ be a triangle center function. 
\begin{enumerate}
    \item $\bX_{\psi} = \bX_{\psi}^{\bPhiA}(0) = \bX_{\psi}^{\bPhiA}(1) = \bX_{\psi}^{\bPhiN}(0) = \bX_{\psi}^{\bPhiN}(1)$.
    \item $\bX_{\psi}^{\bPhiA}\left(\nicefrac{1}{2}\right)$ is the midpoint of $\bX_2$ and $\bX_{\psi}^{\bPhiA}(0)$, inverted through $\bX_2$.
    \item  $\bX_{\psi}^{\bPhiA}\left(\nicefrac{1}{3}\right)$ and $\bX_{\psi}^{\bPhiA}\left(\nicefrac{2}{3}\right)$ related by an inversion through $\bX_2$.
    \item  $\bX_{\psi}^{\bPhiN}\left(\nicefrac{1}{2}\right) = \bX_2$.
\end{enumerate}
\end{proposition}

\subsubsection{Semi-invariant triangle curves}

Next, we examine two features in the study of global properties of triangle curves, beginning with the concept of \emph{semi-invariance}, see \figref{fig:triangle_curves_semi}.

\begin{definition}
If, for all triangles $\Delta \in \triangle$, the curves $t \mapsto \bX_{\psi}^{\bPsi}(t)$ differ only by an affine transformation, we refer to the resulting triangle curves as \emph{semi-invariant}, and the corresponding triangle center $\bX_{\psi}$ as a \emph{semi-invariant curve-generating center} with 
respect to $\bPsi$. 
\end{definition}

Note that this is a non-trivial requirement. For most triangle centers, the connection between the center and the triangle is of higher order (depending on the triangle's side lengths), and the location of triangle centers is not generally preserved under affine transformations. Although it might initially seem counterintuitive that examples of non-trivial curves exist, \secref{sec:mainresult} presents four two-parameter families of triangle centers that are related to two
classical algebraic curves through affine transformations.

The next proposition highlights a consequence of \lemref{lem:combinationsConstantAffine} that the existence of a non-trivial semi-invariant curve-generating center $\bX_{\psi}$ with respect to a triangle family $\bPsi$ implies the existence of a one-parameter family of semi-invariant curve-generating centers with respect to the same family. Equivalently, this also implies the existence of a one-parameter family of triangle families $\bPsi_{\sigma}(t)$ for which $\bX_{\psi}$ is a semi-invariant curve-generating center.

\begin{proposition}\label{prop:ScalingCentersAndTrafos}
For $\psi\in\cT$, let $\bX_{\psi}$ be a semi-invariant curve-generating center with respect to the family $\bPsi(t)$, and let $\psi_{\sigma}$ correspond to the following constant-affine combination of $\bX_{\psi}$ and the centroid $\bX_2$, that is,  
$$
\psi_{\sigma}(a,b,c) = (1-\sigma)\frac{\psi(a,b,c)}{\nmlz{\psi}(a,b,c)} + \sigma \frac{\phi_2(a,b,c)}{\nmlz{\phi_2}(a,b,c)}.
$$
for $\sigma\in\RR\backslash\{0,1\}$.

\begin{itemize}
\item  If $\nmlz{\psi_{\sigma}}\neq 0$, the triangle centers corresponding to $\psi_{\sigma}$ 
are semi-invariant curve-generating centers, with
\begin{equation}\label{eqn:psisigma}
\bX_{\psi_{\sigma}}^{\bPsi}(t) = (1-\sigma) \bX_{\psi}^{\bPsi}(t) + \sigma \bX_2.
\end{equation}
\item Equivalently, $\bX_{\psi}$ is a semi-invariant curve-generating triangle center function with respect to the triangle families $\bPsi_{\sigma} = \bPhiS(\sigma)\circ \bPsi$, where $\bPhiS(\sigma)$ denotes the scaling w.r.t.\ $\bX_2$ by the scalar factor $\sigma \in \RR\backslash\{0,1\}$ as described in Equation~\eqref{eqn:scaling}.
\end{itemize}
The resulting one-parameter families of semi-invariant triangle curves relate by
\begin{equation*}
\bX_{\psi_{\sigma}}^{\bPsi}(t) = \bX_{\psi}^{\bPsi_{\sigma}}(t).
\end{equation*}
Additionally, it follows from \lemref{lem:essDifferentConstAffine} that if $\bX_{\psi}$ is essentially different from $\bX_2$, the triangle centers $\bX_{\psi_{\sigma}}$ are essentially different from $\bX_{\psi}$ and $\bX_2$, and for $\sigma \neq \overline{\sigma}$, $\bX_{\psi_{\sigma}}$ is essentially different from $\bX_{\psi_{\overline{\sigma}}}$. 
\end{proposition}

Note the importance of combining two normalized triangle center functions in order to ensure that the coefficients of the linear combination in \teqref{eqn:psisigma} do not depend on the shape of the triangle (unlike in \teqref{eqn:psiomega}). In the next subsection, we use \propref{prop:ScalingCentersAndTrafos} to  extend families of semi-invariant triangle centers by an additional parameter.
To this end, we define:

\begin{definition}\label{def:scaledfamilies}
Given a family of triangle-center functions $\Omega$, we define its \emph{scaled family} as
\begin{equation}\label{eqn:OmegaSigma}
\Omega_{\sigma} = \left\{ \psi_{\sigma} \mid \psi \in \Omega, \sigma \in \RR\right\}.
\end{equation}
\end{definition}

The above \propref{prop:ScalingCentersAndTrafos} shows that semi-invariance of a center with respect to a nontrivial triangle family automatically extends to a one-parameter family of triangle families. 
 In the decomposable case, this observation can be strengthened further:

\begin{lemma}\label{lem:semiInvariantCombo}
    If $\bX_{\psi}$ is semi-invariant with respect to a triangle family $\bPsi$, then it remains semi-invariant with respect to $\bPhiS(\sigma(t)) \circ \bPsi(\tau(t))$ for arbitrary scalar-valued functions $\sigma(t)$ and $\tau(t)$.

\end{lemma}

\begin{proof}
If $\bX_{\psi}$ is a semi-invariant curve-generating center with respect to $\bPsi(t)$, the corresponding curve allows a parametrization of the form 
$$
\bX_{\psi}^{\bPsi}(t) = \bX_2 + l_x(t) \bV_x + l_y(t) \bV_y, 
$$
where the two vectors $\bV_x$ and $\bV_y$  correspond to sheared axes of the unsheared curve $(l_x(t), l_y(t))$ and depend the triangle's configuration.

Changing the parametrization of $\bPsi$ affects the triangle center curve as follows
$$
\bX_{\psi}^{\bPsi(\tau)}(t) = \bX_2 + l_x(\tau(t)) \bV_x + l_y(\tau(t)) \bV_y.
$$
Since this curve still corresponds to an unsheared curve $(l_x(\tau(t)), l_y(\tau(t)))$ that does not depend on the triangle's configuration, the triangle center $\bX_{\psi}$ is still semi-invariant with respect to $\bPsi(\tau(t))$.

Concatenating the triangle family with a scale family with scale factor $\sigma(t)$ changes the triangle curve to 
$$
\bX_{\psi}^{\bPhiS(\sigma)\circ \bPsi(\tau)}(t) = \bX_2 +  \sigma(t)\left(l_x(\tau(t)) \bV_x + l_y(\tau(t)) \bV_y\right).
$$
This curve corresponds to the unsheared curve $\sigma(t)\left(l_x(\tau(t)), l_y(\tau(t))\right)$  and consequently the triangle center $\bX_{\psi}$ is a semi-invariant with respect to $\bPhiS(\sigma(t))\circ\bPsi(\tau(t))$.
\end{proof}

Note that, for a generic semi-invariant curve-generating center, this lemma significantly broadens the class of triangle families with respect to which the center remains semi-invariant. The following corollary highlights a special case that follows directly from \lemref{lem:decompositionSigmaTau}.

\begin{corollary}\label{cor:nedianbecausealiquot}
If $\bX_{\psi}$ is a semi-invariant curve-generating center with respect to the aliquot triangle family $\bPhiA$, then it is also semi-invariant with respect to all decomposable triangle families, including the nedian family $\bPhiN$.
\end{corollary}

\subsubsection{Invariant triangle curves}

A particularly intriguing special case of semi-invariant triangle curves are those that do not shear, see  \figref{fig:triangle_curves_inv}.

\begin{definition}
If, for all triangles $\Delta \in \triangle$, the curves  $\bX_{\psi}^{\bPsi}(t)$ differ only by translation, rotation, scaling, or reflection, we refer to the triangle curves as \emph{invariant}, and the corresponding triangle center $\bX_{\psi}$ as an \emph{invariant curve-generating center} with respect to $\bPsi$.
\end{definition}

\begin{remark}
Since $\bPsi$-triangle families preserve the centroid (\lemref{lem:propertiesFamilies}), all curves $t\mapsto \bX^\bPsi_\psi(t)$ are naturally compared after translating so that $\bX_2$ coincides.
\end{remark}

While the existence of invariant curve-generating centers might be even more surprising, we highlight positive results by Mundilova~\cite{Mundilova:2024:MTT}, who demonstrates that $\bX_{13}$ and $\bX_{15}$ are invariant curve-generating centers with respect to the aliquot triangle family. In \secref{sec:mainresult}, we extend these results and provide four one-parameter families of triangle centers that are invariant curve-generating centers with respect to decomposable triangle families.

The properties of semi-invariant curve-generating centers highlighted in \propref{prop:ScalingCentersAndTrafos} and \lemref{lem:semiInvariantCombo} can be strengthen for invariant centers as:

\begin{proposition}\label{prop:ScalingCentersAndTrafosInv}
  If $\Omega_{\inv}$ is a set of triangle centers that are invariant with respect to a triangle family $\bPsi$, then $\Omega_{\inv,\sigma}$ is a set of invariant curve-generating centers with respect to $\bPsi$ as well.
\end{proposition}

An immediate consequence from the special case of the proof of \lemref{lem:semiInvariantCombo}, where the axes $\bV_x$ and $\bV_y$ remain the same up to scaling and rotation due to the properties of invariant curves is summarized as follows.

\begin{corollary}\label{cor:invariantSigmaTau}
   If $\bX_{\psi}$ is an invariant curve-generating triangle center for a triangle family $\bPsi$, then it remains invariant under $\bPhiS(\sigma(t)) \circ \bPsi(\tau(t))$ for arbitrary scalar-valued functions $\sigma(t)$ and $\tau(t)$.
\end{corollary}

 The following \lemref{lem:invariant_concatenation} establishes that invariance is preserved under concatenation of triangle families.

\begin{lemma}\label{lem:invariant_concatenation}
Let $\bX_{\psi}$ be a triangle center, essentially different from $\bX_2$, that is an invariant curve-generating center with respect to two triangle families $\bPsi_0$ and $\bPsi_1$. Then $\bX_{\psi}$ is an invariant curve-generating center with respect to $\bPsi_1 \circ \bPsi_0$. 
\end{lemma}

\begin{proof}
First note that since $\Delta$ is assumed not to be equilateral and $\bX_2$ is essentially different from $\bX_{\psi}$, we can place a local coordinate system composed of the axis, spanned by $\bX_2$ and $\bX_{\psi}$, and its rotation by $\nicefrac{\pi}{2}$.

Since the curves are invariant, there exist functions $r_{0}(t)$, $r_{1}(t)$, $\theta_{0}(t)$, and $\theta_{1}(t)$, such that the curves can be described in a local coordinate system centered at $\bX_2$ as
    $$
    \bX_{\psi}^{\bPsi_i}(t) = r_{i}(t)\, \bR_{\theta_{i}(t)}\cdot \left(\bX_{\psi} - \bX_2\right) +\bX_2 
    $$
    where $\bR_{\theta_i}$ denotes the matrix that rotates a vector by angle $\theta_i$. 
    
    The concatenation $\bPsi_1 \circ \bPsi_0$ is obtained by applying the $\bPsi_1$-triangle family invariant curve construction to the triangle $\triangle_t^{\bPsi_0}$.  Note that the triangle center corresponding to $\psi$ in $\triangle_t^{\bPsi_0}$ is $\bX_{\psi}^{\bPsi_0}(t)$.  It follows: 
\begin{align*}
    \bX_{\phi}^{\bPsi_1 \circ \bPsi_0}(t) &= r_{1}(t)\, \bR_{\theta_{1}(t)}\cdot \left(\bX_{\psi}^{\bPsi_0}(t) - \bX_2\right) +\bX_2\\
     &=  r_{1}(t)\, \bR_{\theta_{1}(t)}\cdot \left( r_{0}(t)\, \bR_{\theta_{0}(t)}\cdot \left(\bX_{\psi} - \bX_2\right)\right) +\bX_2 \\
     &= \left(r_{1}(t)\, r_{0}(t)\right)\,\left( \bR_{\theta_{1}(t)}\cdot \bR_{\theta_{0}(t)}\right)\cdot \left(\bX_{\psi} - \bX_2\right) +\bX_2 \\
     &= \left(r_{0}(t)\, r_{1}(t)\right)\, \bR_{\theta_{0}(t)+\theta_{1}(t)}\cdot  \left(\bX_{\psi} - \bX_2\right) +\bX_2.
\end{align*}
Consequently, the curve corresponding $\bX_{\psi}$ with respect to the concatenation of $\bPsi_0$ and $\bPsi_1$ has a construction that is also independent of the underlying triangle shape, which concludes our argument.

Note that in the special case where $\bPsi_1 = \bPhiS(\sigma)$, as discussed in \lemref{lem:semiInvariantCombo}, the angular function $\theta_1(t) = 0$ and $r_1(t) = \sigma(t)$.
\end{proof}

Upon finding an invariant curve-generating center $\bX_{\psi}$ with respect to a non-trivial triangle family $\bPsi$ (neither scaling and nor identity mapping), in addition to \corref{cor:invariantSigmaTau}, we can find invariant triangle families by $n$-fold concatenation $(\bPsi)^n$ of $\bPsi$ for any $n\in \NN$ by induction.

Note that for triangle centers that are essentially different from $\bX_2$, \lemref{lem:invariant_concatenation} applies to triangle families that can be expressed as a concatenation of triangle families but are not themselves decomposable; therefore, \corref{cor:invariantSigmaTau} is not applicable. An example of such families is given by the aliquot family $\bPhiA$ together with the triangle family
$$
\bPsi_0 \colon t \mapsto \begin{pmatrix}
    1+t\\ 1\\ 1-t+3t^2
\end{pmatrix}.
$$
The concatenation
$$
\bPsi_1 = \bPsi_0 \circ \bPhiA \colon t \mapsto \begin{pmatrix}
    1+t^2\\ 1-t+4t^2-3t^3\\ 1+t-2t^2+3t^3
\end{pmatrix}
$$
is a triangle family that is not decomposable, since $\delta_{\bPsi_2}(t) = 0$ for $t\in\RR$.

\subsection{Semi-invariant curve-generating centers that relate to trisectrices}\label{sec:mainresult}

With these preparations in place, we now highlight four families of triangle centers that are semi-invariant curve-generating centers with respect to decomposable $\bPsi$-families. To this end, we define 
\begin{align*}
\psi_{\Gamma;\lambda_0:\lambda_1}(a,b,c)&\coloneqq\sqrt{3}a(-a^2+b^2+c^2)\lambda_0  +4 a \area \lambda_1 \\
\psi_{\Xi;\lambda_0:\lambda_1}(a,b,c)&\coloneqq-a (-a^2+b^2+c^2) \lambda_0 + a(a^2+b^2+c^2)\lambda_1,
\end{align*}
and name
\begin{align*}
\Gamma &\coloneqq \left\{\psi_{\Gamma;\lambda_0:\lambda_1}(a,b,c) \mid \lambda_0, \lambda_1 \in \RR \right\}, &
\Gamma^{-1} &\coloneqq \left\{\psi_{\Gamma;\lambda_0:\lambda_1}^{-1}(a,b,c) \mid \lambda_0, \lambda_1 \in \RR \right\},\\
\Xi &\coloneqq \left\{ \psi_{\Xi;\lambda_0:\lambda_1}(a,b,c)\mid \lambda_0, \lambda_1 \in \RR\right\},
& \Xi^{-1} &\coloneqq \left\{ \psi_{\Xi;\lambda_0:\lambda_1}^{-1}(a,b,c)\mid \lambda_0, \lambda_1 \in \RR\right\}.
\end{align*}
The two triangle center families $\Gamma$ and $\Xi$ are special constant-affine parametrizations\footnote{It follows from \exref{ex:brocard} that $\Gamma = \{\bX_3, \bX_6\} \cup \spann_{\const}(\bX_3, \bX_6, \bX_{15})$ and $\Xi = \{\bX_3, \bX_6\} \cup \spann_{\const}(\bX_3, \bX_6, \bX_{32})$. In particular, we have that $\{\bX_{15}, \bX_{16}, \bX_{61}, \bX_{62}\} \subset \Gamma$ and $\{\bX_{32}, \bX_{39}\}\subset \Xi$.} of the \emph{Brocard axis}, the set of triangles collinear with the triangle centers $\bX_3$ and $\bX_6$, see \figref{fig:brocard}.
The sets of their isogonal conjugates $\Gamma^{-1}$ and $\Xi^{-1}$ are one-parameter families of triangle centers on the \emph{Kiepert hyperbola}. 

\begin{figure}

\begin{subfigure}{0.47\textwidth}
\centering
\begin{footnotesize}
\begin{overpic}[width=0.8\textwidth]{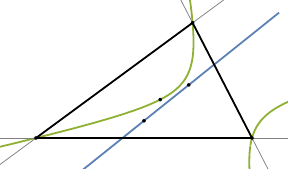}
\put(50,12){$\bX_3$}
\put(65,24){$\bX_6$}
\put(50,26){$\bX_2$}
\put(80,35){Brocard}
\put(84.5,30){axis}
\put(22,53){Kiepert hyperbola}
\end{overpic}%
\vspace{.6cm}
\end{footnotesize}%
    \caption{Brocard axis and Kiepert hyperbola.}\label{fig:brocard}
\end{subfigure}%
\begin{subfigure}{0.49\textwidth}
\centering
\begin{footnotesize}
\hspace{0.4cm}
\begin{overpic}[width=0.8\textwidth]{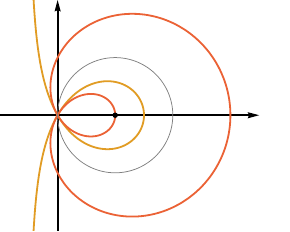}
\put(40,42){$(k,0)$}
\put(43,25){$M$}
\put(77,20){$L$}
\end{overpic}%
\end{footnotesize}%
 \caption{Maclaurin trisectrix $M$ and Lima\c{c}on trisectrix $L$.}\label{fig:maclaurin}
\end{subfigure}

\caption{Illustration of concepts discussed in \thmref{thm:main}. }
\end{figure}
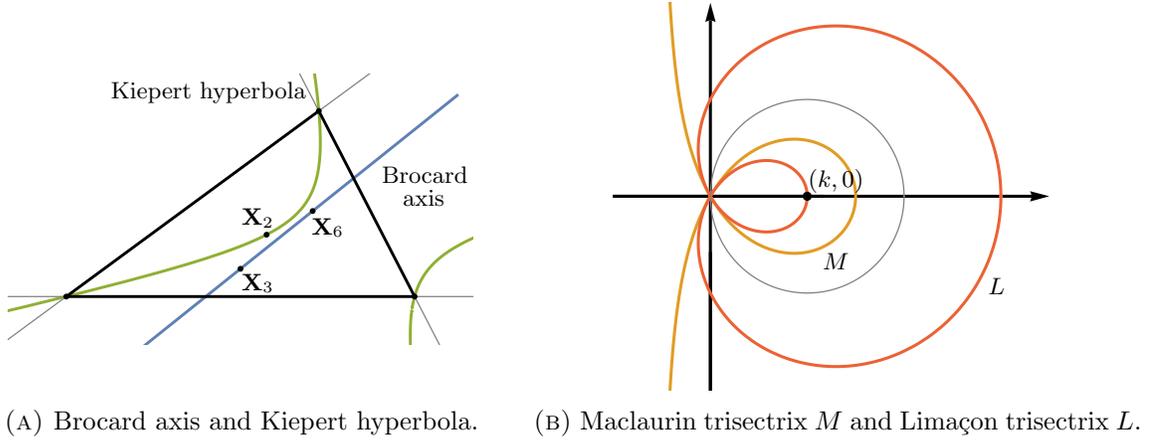

Building on the work of Mundilova~\cite{Mundilova:2024:MTT}, we identify these triangle centers as semi-invariant curve-generating centers through their connection with special algebraic curves. In particular, the \emph{Maclaurin trisectrix} is the plane algebraic curve defined by
$$
2x(x^2+y^2) = k(3x^2-y^2). %\quad \text{and} \quad k^2(x^2+y^2) = (x^2+y^2-2kx)^2.
$$
Applying an inversion with respect to the circle centered at 
$(k,0)$, the \emph{center point}, with radius $k$ transforms this curve into the \emph{Lima\c{c}on trisectrix}, defined by
\begin{equation}\label{eqn:limacon}
k^2(x^2+y^2) = (x^2+y^2-2kx)^2,
\end{equation}
see \figref{fig:maclaurin}.

With these preparations in place, our main result can be stated as follows: 
\begin{theorem}\label{thm:main}
Let $\bPsi$ be a decomposable triangle family. Then, the triangle families 
$$
\Omega \coloneqq \Gamma \cup \Gamma^{-1} \cup \Xi \cup \Xi^{-1}
$$
are semi-invariant triangle curve-generating centers. In particular, we have that:
 \begin{itemize}
     \item If $\bPsi = \bPhiA$, the triangle curves are sheared Maclaurin trisectrices.
     \item If $\bPsi = \bPhiN$, the triangle curves are sheared Lima\c{c}on trisectrices.
 \end{itemize}
 Furthermore, among these triangle center functions, only the triangle centers $$
 \Omega_{\operatorname{inv}} \coloneqq \left\{\bX_{13}, \bX_{14}, \bX_{15}, \bX_{16}\right\},$$ 
 are invariant.
\end{theorem}

\begin{proof}
The proof of this theorem builds on the work of Mundilova~\cite{Mundilova:2024:MTT}, who shows that the triangle centers $\bX_{13}$ and $\bX_{15}$ generate invariant curves by establishing a parametric relationship between the Maclaurin trisectrix and the triangle curve traced with respect to the aliquot triangle family, using algebraic simplifications carried out in Mathematica. 

As the proof is applicable for all all triangle centers in $\Omega$, in the following, we will use $\psi$ to represent either of the triangle center functions in $\psi_{\Gamma;\lambda_0:\lambda_1}$, $\psi_{\Xi;\lambda_0:\lambda_1}$, $\psi_{\Gamma;\lambda_0:\lambda_1}^{-1}$, or $\psi_{\Xi;\lambda_0:\lambda_1}^{-1}$.

We structure our proof in five steps:
\begin{itemize}
    \item \emph{Step 1 (Preparations):}
We assume without loss of generality that a given triangle $(a,b,c)$ is placed such that its side $c$ coincides with the $x$-axis, namely,
\begin{equation}\label{eqn:triangle}
    \bA = (0,0), \quad \bB = (c,0), \quad \text{and} \quad 
    \bC = \left(\frac{1}{2c}\left(-a^2+b^2+c^2\right), \frac{2}{c}\area\right).
\end{equation}    
We determine the parametrization $\bX_{\psi}^{\bPhiA}(t)$ of a triangle curve corresponding to a triangle center function $\psi$ and the aliquot triangle family. 

\item \emph{Step 2 (Test curve):} Since the algebraic expressions for the triangle curves $\bX_{\psi}^{\bPhiA}(t)$ do not easily simplify, we follow the approach of Mundilova~\cite{Mundilova:2024:MTT}, and use the parametrization of the Maclaurin trisectrix
\begin{equation}\label{eqn:maclaurint}
\bM(t) = (M_x(t), M_y(t)) = \frac{3k}{2 (1 - 3(1-t)t)} \left((1-t)t, \sqrt{3}(1-t)t(1-2t)\right)
\end{equation}
to algebraically compare them with.

\begin{figure}
\centering
\begin{footnotesize}
\begin{overpic}[width=0.6\textwidth]{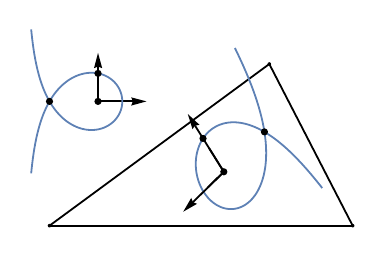}
\put(70,37){$\bX_{\psi}$}
\put(58,22){$\bX_2$}
\put(74.5,20){$\bT_{\psi}(t)$}
\put(54,33){$\bX_{\psi}^{\bPhiA}\!\left(\!\tfrac{1}{3}\!\right)$}
\put(21,40){$(k,0)$}%
\put(4,42.75){$\bM(0)$}
\put(27,52){$\bM(\nicefrac{1}{3})$}
\end{overpic}%
\end{footnotesize}%
\vspace{-0.6cm}
    \caption{Illustration of main idea of the proof of \thmref{thm:main}.}\label{fig:ill_main}
\end{figure}
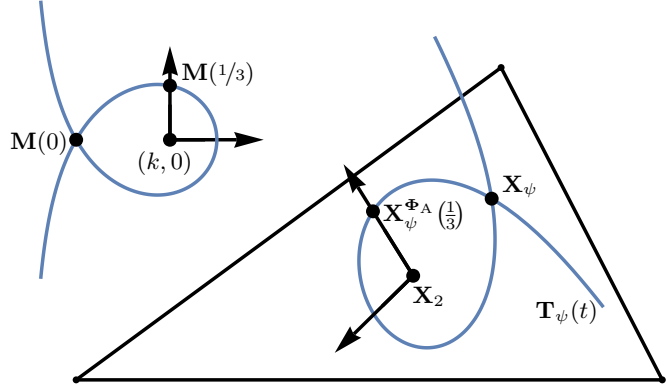

Specifically, we define sheared trisectrices $\bT_{\psi}(t)$ based on local properties of triangle curves discussed in \propref{prop:propAliquotNedian}, obtained by mapping the points $\bM(0) = (0,0)$, the center $(k,0)$, and $\bM(\nicefrac{1}{3}) = (k,\nicefrac{k}{\sqrt{3}})$ to $\bX_{\psi} = \bX_{\psi}(0)$, the centroid $\bX_2$, and $\bX_{\psi}(\nicefrac{1}{3})$, respectively (see \figref{fig:ill_main}). 
To this end, we represent 
\begin{equation}\label{eqn:TpsibPhi}
\bT_{\psi}(t) = \bX_{\psi} + l_{x}(t) \bV_{\psi,x} + l_{y}(t) \bV_{\psi,y},
\end{equation}
where $\bV_{\psi,x}$ and $\bV_{\psi,y}$ are two vectors corresponding to the $x$ and $y$-axes of the unsheared trisectrix, 
namely, 
$$
\bV_{\psi,x} = \bX_2 - \bX_{\psi}
\quad \text{and} \quad
\bV_{\psi,y} = \bX_{\psi}^{\bPhiA}\left(\frac{1}{3}\right) - \bX_2,
$$
 and $l_x(t)$ and $l_y(t)$ two scalar-valued functions that encode the shape and scale of the sheared trisectrix, 
\begin{equation}\label{eqn:PhiAins}
l_x(t) = M_x(t),
\quad 
\text{ and }
\quad 
l_y(t) = \frac{M_y(t)}{M_y(\nicefrac{1}{3})} = \sqrt{3} \,M_y(t).
\end{equation}

\item \emph{Step 3 (Algebraic verification):} 
We use the advanced computer algebra system \emph{Mathematica} to show that the four equations 
$$
\bX_{\psi_{\Gamma}}^{\bPhiA}(t)== \bT_{\psi_{\Gamma}}^{\bPhiA}(t), 
 \quad 
\bX_{\psi_{\Xi}}^{\bPhiA}(t)== \bT_{\psi_{\Xi}}^{\bPhiA}(t), 
\quad
\bX_{\nicefrac{1}{\psi_{\Gamma}}}^{\bPhiA}(t)== \bT_{\psi_{\Gamma}^{-1}}^{\bPhiA}(t), 
 \quad \text{and} \quad
\bX_{\nicefrac{1}{\psi_{\Xi}}}^{\bPhiA}(t)== \bT_{\psi_{\Xi}^{-1}}^{\bPhiA}(t)
$$ 
simplify to true. 
Specifically, we use the \emph{FullSimplify} command with the assumptions 
$$
\{a > 0, b > 0, c > 0, a + b > c, b + c >a, c + a>0, \text{Element}[t, \text{Reals}]\}.
$$
The attached notebook confirms that the triangle centers in $\Gamma$, $\Xi$, $\Gamma^{-1}$, and $\Xi^{-1}$ are semi-invariant.  

\item \emph{Step 4 (Invariant curve-generating centers): } Among all triangle center functions in $\Omega$, we identify those that correspond to unsheared and properly scaled Maclaurin trisectrices. This property can be obtained by searching for combinations of scalars $\lambda_0$, $\lambda_1$ that are independent of the triangle side lengths, such that $\bV_{\psi,x} \cdot \bV_{\psi,y} = 0$ and $|\bV_{\psi,x}|^2 = 3 |\bV_{\psi,y}|^2$. The attached Mathematica notebook presents the computations, which can be summarized as follows:
\begin{itemize}
\item \emph{Triangle centers in $\Gamma$}: Both conditions are satisfied for $\lambda_0 : \lambda_1 \in \{1:1, -1:1\}$. These two solutions correspond to the triangle centers $\bX_{15}$ and $\bX_{16}$, respectively.

\item \emph{Triangle centers in $\Gamma^{-1}$}: Both conditions are satisfied for the combinations $\lambda_0 : \lambda_1 \in \{0:1, 1:1, -1:1\}$. These correspond to the triangle centers $\bX_6 = \bX_{2}^{-1}$, $\bX_{13} = \bX_{15}^{-1}$, and $\bX_{14} = \bX_{16}^{-1}$, respectively.

\item \emph{Triangle centers in $\Xi$}: In this case, we cannot find solutions to the two equations that are independent of the triangle's shape. Consequently, this constant-affine parametrization of the Brocard axis admits no invariant triangle centers.

\item \emph{Triangle centers in $\Xi^{-1}$}: The only invariant triangle center corresponds to $\lambda_0: \lambda_1 = 0:1$, and is therefore the centroid $\bX_2 = \bX_6^{-1}$. The other two solutions depend on the shape of the triangle.
\end{itemize}
    
\item \emph{Step 5 (Generalization of results):} Using \corref{cor:nedianbecausealiquot}, we conclude that the triangle centers in $\Omega$ and $\Omega_{\inv}$ are semi-invariant or invariant, respectively, with respect to all decomposable triangle families. 

    To show the stated connection between the nedian family and Liman\c{c}on trisectrix, we use the functions $\tau(t)$ and $\sigma(t)$ in \corref{cor:nediantoaliquot} to find   
\begin{align}\label{eqn:limacont}
\bL(t) &= \sigma(t) \left(\bM(\tau(t)) - (k,0)\right) + (k,0) \nonumber\\
&= \frac{3k}{2(1-(1-t)t)^2} \left(t(1+t -2(2-t)t^2), 
\sqrt{3}(1-t)t(1-2t)\right).
\end{align}
Using \teqref{eqn:limacon}, it is straightforward to verify that $\bL(t)$ parametrizes a Lima\c{c}on trisectrix. 
\end{itemize}%
\end{proof}

\begin{remark}
    The found parametrizations of the Maclaurin (\teqref{eqn:maclaurint}) and Lima\c{c}on trisectrix (\teqref{eqn:limacont}) are not related by a circular inversion. 
\end{remark}

Finally, with \propref{prop:ScalingCentersAndTrafos} and \propref{prop:ScalingCentersAndTrafosInv}, we extend our results to four two-parameter families of semi-invariant triangle centers and four one-parameter families of invariant triangle centers  as follows:

\begin{corollary}
The families of scaled triangle centers $\Omega_{\sigma}$ (see~\defref{def:scaledfamilies}) are semi-invariant curve-generating centers. Furthermore, the scaled triangle families $\Omega_{\inv,\sigma}$ are invariant curve-generating centers.
\end{corollary}

\subsection{Triangle families that generate target curves}\label{sec:inverseProblem}

Upon successfully identifying triangle centers that generate semi-invariant curves with respect to decomposable triangle families in \thmref{thm:main}, a natural question arises: Can this process be reversed? Specifically, given a triangle center, how can we determine a triangle family for which a particular curve arises?

The next theorem presents an explicit construction of a triangle family which, together with a triangle center $\psi\in\Omega_{\sigma}$, generates a given target curve:

\begin{theorem}\label{thm:inverse}
    Let a $\bT(t)$ be a target curve, defined for $t \in T = \left(-\frac{3\pi}{2},\frac{3\pi}{2}\right)$, and given in polar coordinates as
$$
\bT(t) = r(t) \left(\cos \theta(t), \sin \theta(t)\right),
$$
with two scalar valued functions $r(t)$ and $\theta(t)$, with $r(t) \neq 0$ for all $t\in T$.

Consider the triangle family $t \mapsto \bPsi(t) = \left(\bPsi_1(t),\bPsi_2(t),\bPsi_3(t)\right)$, where
\begin{align*}
    \bPsi_1(t) &= 1 + r(t) \left(\cos\frac{\theta(t)}{3} + \sqrt{3}\sin\frac{\theta(t)}{3}\right), \\
    \bPsi_2(t) &= 1 + r(t) \left(\cos\frac{\theta(t)}{3} - \sqrt{3}\sin\frac{\theta(t)}{3}\right), \\
    \bPsi_3(t) &= 1 - 2 r(t) \cos\frac{\theta(t)}{3}.
\end{align*}

For a triangle center $\psi\in\Omega_{\sigma}$, the associated triangle curves $\bX_{\psi}^{\bPsi}(t)$ are sheared versions of the target curve $\bT(t)$.
If $\psi\in\Omega_{\inv,\sigma}$, the resulting curves are unsheared.
\end{theorem}

\begin{proof}
To compute the corresponding triangle family, we must first determine the relationship between the parametrization of the Maclaurin trisectrix given in \teqref{eqn:maclaurint} and the target curve, involving a reparametrization and a scaling transformation with respect to the center $(k,0)$. Specifically, we compute the two scalar-valued functions $\sigma(t)$ and $\tau(t)$, such that     
$$
\bT(t) = \sigma(t)(\bM(\tau(t))-(k,0))+(k,0).
$$
Since every line incident with $(k,0)$ intersects the Maclaurin trisectrix (in general) at three points, the equation admits three solutions. Among these, it is sufficient to consider a single solution, such as
$$
\sigma(t) = - r(t) \left(\cos \frac{\theta(t)}{3} + \sqrt{3} \sin \frac{\theta(t)}{3}\right) \quad \text{and}\quad 
\tau(t) = - \frac{1}{6} \left(3 + \frac{3 - 2 \sqrt{3}\sin \frac{2\theta(t)}{3}}{-1 + 2 \cos \frac{2\theta(t)}{3}}\right),
$$
where $t\in T$.
Inserting these expressions into \teqref{eqn:PhiSigmaTau} yields the formulas stated above. 

Finally, we note that $\Psi_1(t) + \Psi_2(t) +\Psi_3(t) = 3 \neq 0$ and $\Psi_1(t) = \Psi_2(t) = \Psi_3(t)$ only if $r(t) = 0$. Consequently, $\bPsi(t)$ is a triangle family.
\end{proof}

\begin{figure}
\centering
\includegraphics[width=0.32\textwidth, trim = 5.3cm 2.cm 4.9cm 5cm,clip]{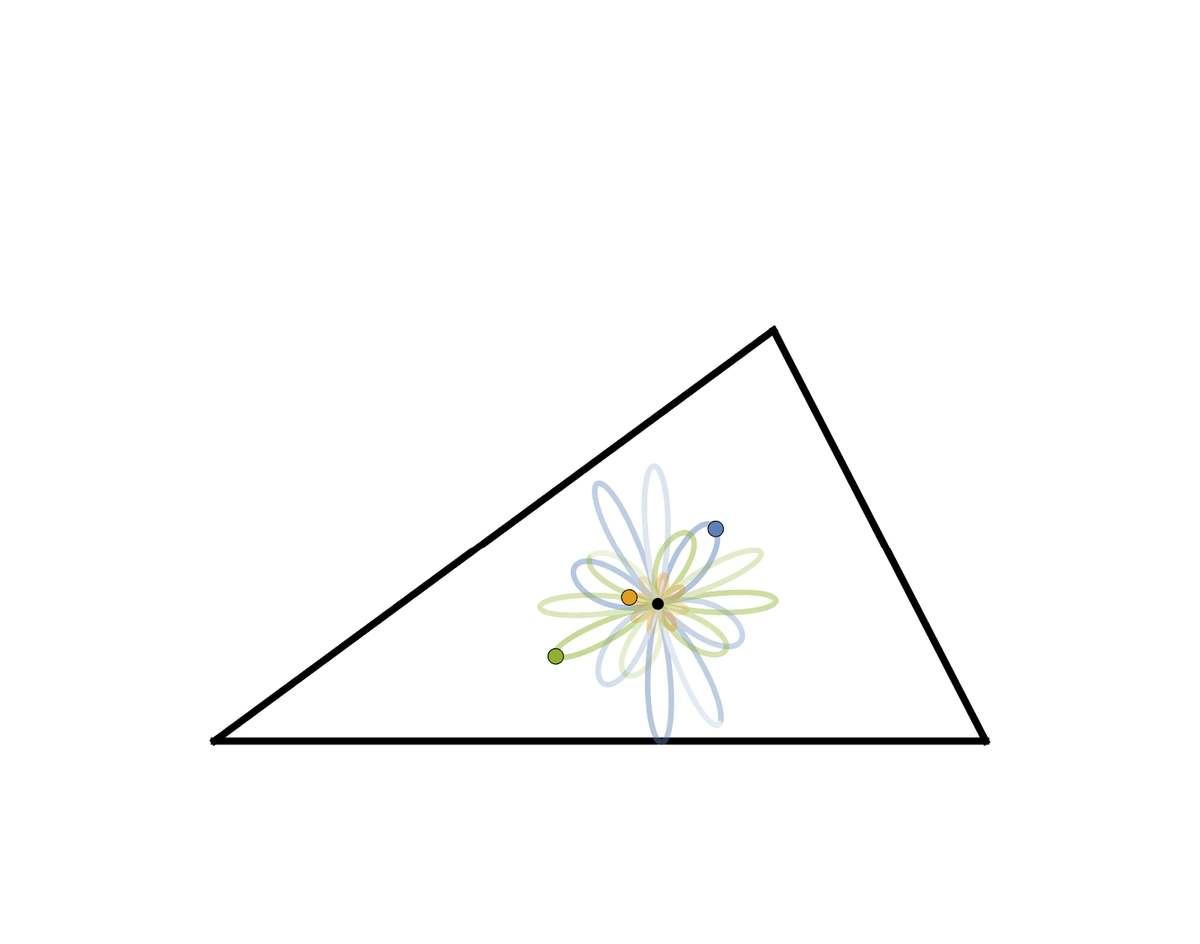}
\hfill
\includegraphics[width=0.32\textwidth, trim = 5.3cm 2.cm 4.9cm 5cm,clip]{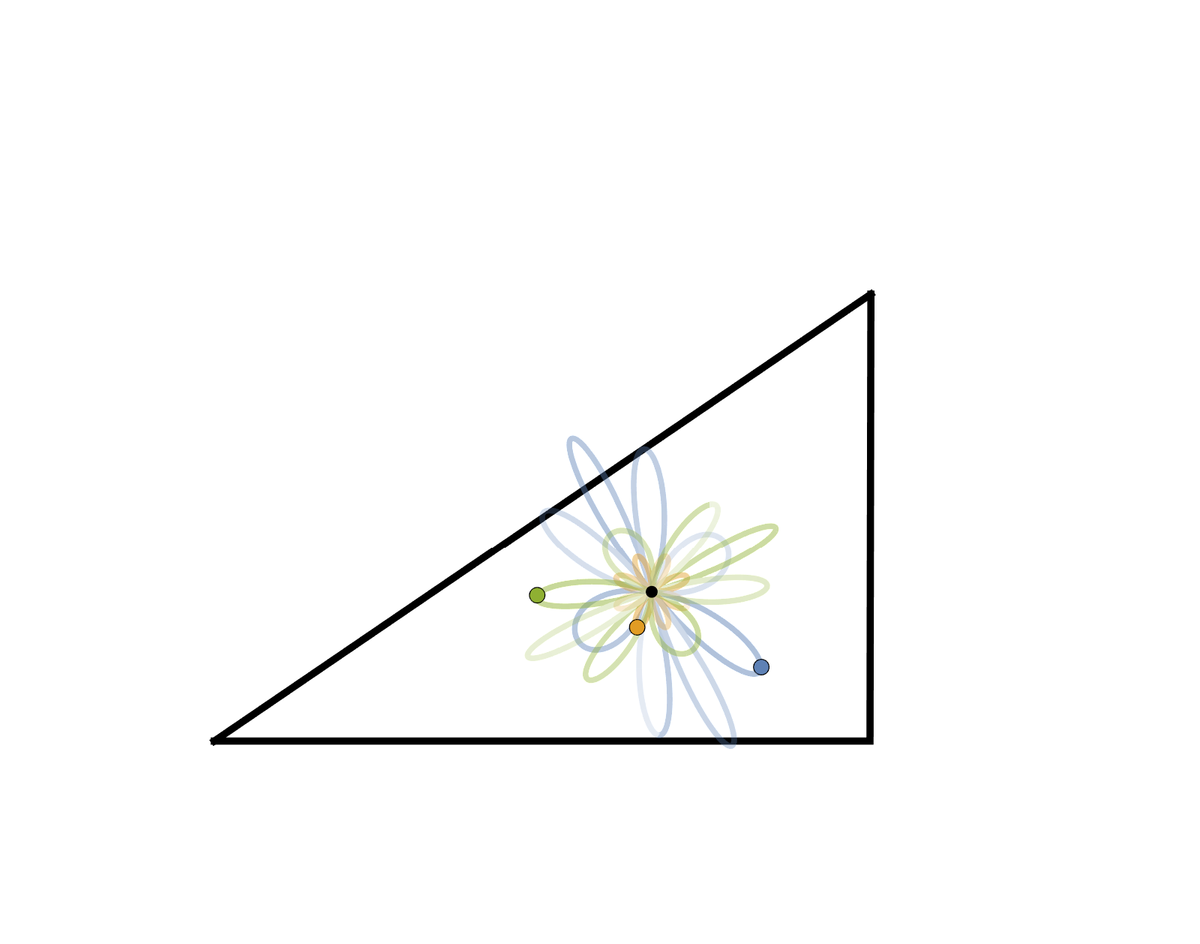}
\includegraphics[width=0.32\textwidth, trim = 5.3cm 2.cm 4.9cm 5cm,clip]{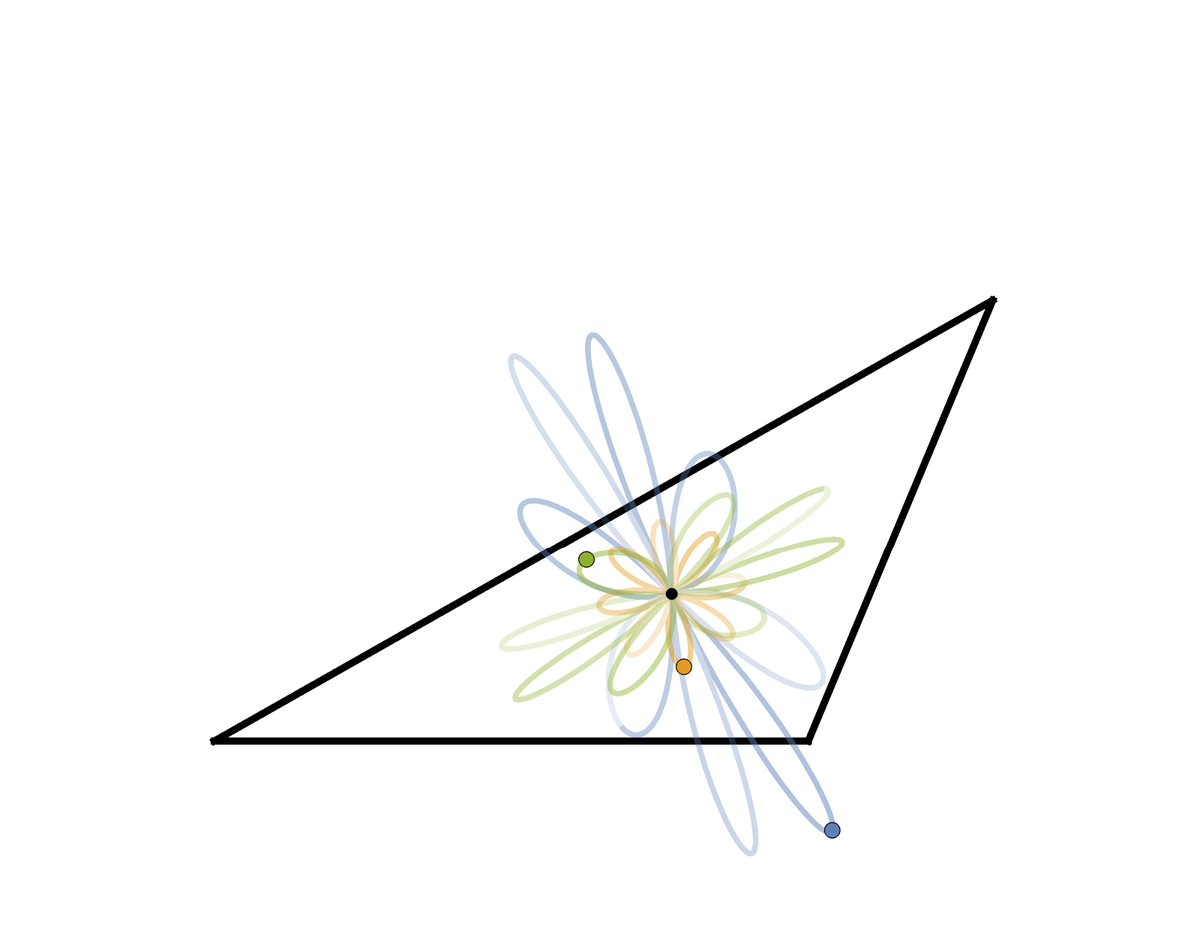}
\caption{Illustration of the outcome of our ``inverse design problem'': identifying triangle families that trace rose curves with the semi-invariant curve-generating centers $\bX_3$ (blue) and $\bX_6$ (green) and the invariant curve-generating center $\bX_{13}$ (yellow). }\label{fig:inverse}
\end{figure}

\begin{example}
    
\thmref{thm:inverse} allows us to find triangle families that, in combination with $\psi\in\Omega$, generate sheared \emph{rose curves}, whose parametrization in polar coordinates is given by 
    $$
    r(t) = a \cos (n t) \quad \text{and} \quad \theta(t) = t.
    $$
    The resulting triangle family simplifies to  
    \begin{align*}
    \bPhiR \colon t \mapsto  \begin{pmatrix} 
    1+a \cos(nt) \left(\cos \frac{t}{3} + \sqrt{3}\sin\frac{t}{3}
    \right)     \\
     1+a \cos(nt) \left(\cos \frac{t}{3} - \sqrt{3}\sin\frac{t}{3}
    \right) \\
       1-2 a \cos(nt) \cos \frac{t}{3} 
    \end{pmatrix}.
    \end{align*}
\figref{fig:inverse} shows an example with three triangle centers tracing the target rose curve with $a = 1$ and $n = 4$. 
\end{example}

\section{Conclusion}

Inspired by Kimberling's definition of triangle centers, we introduced a framework for defining triangle families, the so-called $\bPsi$-triangle families, using triples of functions. We demonstrated that these families exhibit a rich and intriguing structure. We then investigated the curves traced by triangle centers along $\bPsi$-triangle families. In doing so, we identified general properties of these curves and systematically searched for triangle centers that generate semi-invariant or fully invariant curves. We hope that the ideas and results presented in this work will inspire further exploration into the geometry of triangles and outline some  potential future research directions.

Beyond the properties of semi-invariance and invariance, other aspects of triangle curves may prove interesting, such as potential relationships between curves traced by pairs of triangle centers, such as isogonal conjugates, independent of their (semi-)invariance. Our work on semi-invariant curves builds upon the relationship between the Maclaurin trisectrix and the aliquot triangle families. The exploration of triangle centers that generate other (algebraic) triangle curves with respect to the aliquot family is left for future work. 

Ultimately, a complete classification of triangle center functions with respect to their traced curves could be a long-term goal.  Consequently, the development of more sophisticated theoretical methods would be desirable, both to provide deeper geometric insights and to support the classification of triangle centers based on their relationship to semi-invariant and invariant triangle curves.

\section*{Acknowledgements}
We thank the current and past members of the Geometric Computing Laboratory at EPFL, in particular Prof.~Mark Pauly, for insightful discussions.
Klara Mundilova was supported by the Swiss Government Excellence Scholarship of the Swiss Confederation.

\bibliographystyle{plainnat} % required for using natbib
\bibliography{sample}

@string {FG = "Forum Geom."}

@string {JOG = "J.\ Geom."}

@string {JGG = "J.\ Geometry Graphics"}

@string {KOG = "KoG"}

@string {MCS = "Math.\ Comput.\ Sci."}

@string {MG = "Math.\ Gaz."}

@string { RHAZU = "Rad HAZU, Matemati{\v{c}}ke znanosti"}

@string {ICGG = "Proceedings of the 21st ICGG"}

@inproceedings{Mundilova:2024:MTT,
  title={\href{https://doi.org/10.1007/978-3-031-71225-8_12}{Maclaurin Trisectrices as t-Affine Loci of the First Isogonic and Isodynamic Centers}},
  author={Mundilova, K.},
  booktitle=ICGG,
  pages={136--151},
  year={2024},
  organization={Springer}
}

@article{Dirnbock:2003:Curves,
  title={\href{https://www.heldermann-verlag.de/jgg/jgg07/jgg0702.pdf}{Curves related to triangles: The Balaton-Curves}},
  author={Dirnb{\"o}ck, H. and Schoi{\ss}engeier, J.},
  journal=JGG,
  volume={7},
  pages={23--39},
  year={2003}
}

@article{Abu:2007:ACA,
  title={\href{https://www.heldermann-verlag.de/jgg/jgg11/j11h1abus.pdf}{Another cubic associated with a triangle}},
  author={Abu-Saymeh, S. and Hajja, M. and Stachel, H.},
  journal=JGG,
  volume={11},
  number={1},
  pages={15--26},
  year={2007}
}

@inproceedings{Abu:2005:TCL,
  title={\href{https://www.researchgate.net/profile/Mowaffaq-Hajja/publication/251452056_Triangle_Centers_with_Linear_Intercepts_and_Linear_Subangles/links/00b7d5296d8e9de4b1000000/Triangle-Centers-with-Linear-Intercepts-and-Linear-Subangles.pdf}{Triangle centers with linear intercepts and linear subangles}},
  author={Abu-Saymeh, S. and Hajja, M.},
  booktitle= FG,
  volume={5},
  pages={33--36},
  year={2005}
}

@article{Kodrnja:2017:LVN,
  title={\href{https://hrcak.srce.hr/clanak/450110}{The Loci of Vertices of Nedian Triangles}},
  author={I. Kodrnja and H. Koncul},
  journal=KOG,
  volume={21},
  number={21},
  pages={1--6},
  year={2017}
}

@article{Kodrnja:2023:LCT,
  title={\href{https://doi.org/10.31896/k.27.4}{Locus Curves in Triangle Families}},
  author={I. Kodrnja and H. Koncul},
  journal=KOG,
  volume={27},
  number={27},
  pages={35--42},
  year={2023}
}

@article{Cundy:1995:SCC,
  title={\href{https://doi.org/10.1007/BF01224039}{Some cubic curves associated with a triangle}},
  author={Cundy, H. M. and Parry, C. F.},
  journal=JOG,
  volume={53},
  pages={41--66},
  year={1995},
  publisher={Springer}
}

@article{Odehnal:2011:PLT,
  title={\href{https://heldermann-verlag.de/jgg/jgg15/j15h1odeh.pdf}{Poristic loci of triangle centers}},
  author={Odehnal, B.},
  journal=JGG,
  volume={15},
  number={1},
  pages={45--67},
  year={2011}
}

@article{Satterly:1956:NNT,
  title={\href{https://doi.org/10.2307/3609672}{The nedians, the nedian triangle and the aliquot triangle of a plane triangle}},
  author={Satterly, J.},
  journal=MG,
  volume={40},
  number={332},
  pages={109--113},
  year={1956},
  publisher={Cambridge University Press}
}

@article{Jurkin:2022:LOC,
  title={\href{https://doi.org/10.21857/94kl4cl6wm}{Loci of centers in pencils of triangles in the isotropic plane}},
  author={Jurkin, E.},
  journal=RHAZU,
  number={551=26},
  pages={155--169},
  year={2022},
  publisher={Hrvatska akademija znanosti i umjetnosti}
}

@book{Kimberling:1998:TCC,
    author = {C. Kimberling},
    title = {Triangle Centers and Central Triangles},
    publisher = {Utilitas Mathematica Publishing, Inc.},
    year = {1998}
}

@article{Narboux:2016:TCV,
  title={\href{https://doi.org/10.1007/s11786-016-0254-4}{Towards a certified version of the encyclopedia of triangle centers}},
  author={Narboux, Julien and Braun, David},
  journal=MCS,
  volume={10},
  pages={57--73},
  year={2016},
  publisher={Springer}
}

@misc{Kimberling:2025:ETC,
  author = {Kimberling, C.},
  title = {Encyclopedia of Triangle Centers},
  url = {https://faculty.evansville.edu/ck6/encyclopedia/ETC.html},
  urldate = {2025-02-6}
}

\newpage
\appendix

\section{Supplemental Material for \secref{sec:triangleCenters}}\label{sec:app:section2}

\begin{proof}[Proof of \lemref{lem:equivCenterFunctions}]
   The claim follows directly from 
   \begin{align*}
       \bX_{\psi_1} &= [f(a,b,c)\, \psi_0(a,b,c) : f(b,c,a) \,\psi_0(b,c,a) : f(c,b,a) \,\psi_0(c,a,b)]\\
       &= [f(a,b,c) \,\psi_1(a,b,c) : f(a,b,c)\,\psi_0(b,c,a) : f(a,b,c) \,\psi_0(c,a,b)] \\
       &= [ \psi_0(a,b,c) : \psi_0(b,c,a) : \psi_0(c,a,b)] \\
       &= \bX_{\psi_0}
   \end{align*}
   since $f(a,b,c)\neq 0$.
\end{proof}

\begin{example}[Traceable triangle centers]\label{ex:traceableTriangleCenters} Examples of traceable triangle centers are given by the incenter \(\bX_1\), centroid \(\bX_2\), the circumcenter $\bX_3$, the orthocenter $\bX_4$, and the symmedian point $\bX_6$, as for the triangle center functions given in \tblref{tab:trianglecenters}, the corresponding traces simplify to
\begin{align}
\nmlz{\phi_1}(a,b,c) &= a + b + c \neq 0 &
 \nmlz{\phi_2}(a,b,c) &= 3 \neq 0, \nonumber \\ 
 \nmlz{\phi_3}(a,b,c) &= (4 \area)^2 \neq 0, &
 \nmlz{\phi_6}(a,b,c) &= a^2 +b^2 + c^2 \neq 0. \label{eqn:trace36}
\end{align}

The triangle centers $\bX_{13}$, $\bX_{14}$, $\bX_{15}$, and $\bX_{16}$ are also traceable. However, the computations are more involved and are presented in the accompanying Mathematica notebook. 
At this point, however, we note that for computing the trace, the corresponding triangle center functions in \tblref{tab:trianglecenters} can be expressed in terms of the triangle side lengths $a$, $b$, and $c$ using the law of cosines, resulting in
\begin{align}\label{eqn:phi15abc}
\phi_{15}(a,b,c) & = \frac{1}{4bc}\left(\sqrt{3}(-a^2 + b^2 + c^2) + 4 \area  \right), \\
\phi_{16}(a,b,c) & = \frac{1}{4bc}\left(\sqrt{3}(-a^2 + b^2 + c^2) - 4 \area  \right).
\end{align}
Since the two isodynamic centers are isogonic conjugates of the isogonic centers, it follows that $\phi_{13} = \phi^{-1}_{15}$ and $\phi_{14} = \phi^{-1}_{16}$.
\end{example}

\begin{example}[Not traceable triangle centers]\label{ex:nottraceableTriangleCenters}
On the other hand, the Feuerbach point $\bX_{11}$ is an example of a non-traceable triangle center. 
Its triangle center function is listed in \tblref{tab:trianglecenters} and vanishes for equilateral triangles, \ie, \(a=b=c\), because of the factor $b - c$. Consequently, $\nmlz{\phi_{11}}(a,a,a) = 0$. 

A general family of examples of non-traceable triangle centers is obtained as differences of two distinct bi-symmetric homogeneous polynomials of the same degree, \eg, 
\begin{equation}\label{eqn:notphi1}
\psi(a,b,c)=a^2 + b^2+c^2 - (ab+ bc + ca).
\end{equation}
These have zeros whenever the triangle is equilateral, consequently $\nmlz{\psi}(a,a,a) = 0$, and hence \(\psi\) is not traceable\footnote{Since  the triangle center function stated in Equation~\eqref{eqn:notphi1} is cyclic, one might expect that it is equivalent to the triangle center function $\phi_1(a,b,c) = 1$ of the incenter. However, the factor $f(a,b,c) = \nicefrac{\psi(a,b,c)}{\phi_1(a,b,c)} = \psi(a,b,c)$ vanishes for equilateral triangles,  $\psi(a,a,a) = 0$, and thus $f\not\in\cycT$ and therefore $\psi\not\cong\phi_1$.}.

\end{example}

\begin{example}[Two not essentially different triangle centers]\label{ex:essentiallyDifferent} With  $\bX_3$ and $\bX_{63}$, we find a pair of triangle centers that are not essentially different. Their corresponding triangle center functions are given by
$$
\phi_3(a,b,c) = a(-a^2 + b^2 + c^2) \quad \text{and} \quad
\phi_{63}(a,b,c) = -a^2 +b^2 + c^2.
$$

To determine triangle configurations where the triangle centers coincide, we require their corresponding barycentric coordinates to be identical. Consequently, we consider the following system of equations.
$$
\frac{a \,\phi_3(a,b,c)}{\nmlz{\phi_3}(a,b,c)} = \frac{a \,\phi_{63}(a,b,c)}{\nmlz{\phi_{63}}(a,b,c)}, 
\quad 
\frac{b\,\phi_3(b,c,a)}{\nmlz{\phi_3}(a,b,c)} = \frac{b\,\phi_{63}(b,c,a)}{\nmlz{\phi_{63}}(a,b,c)}, 
\quad 
\frac{c\,\phi_3(c,a,b)}{\nmlz{\phi_3}(a,b,c)} = \frac{c\,\phi_{63}(c,a,b)}{\nmlz{\phi_{63}}(a,b,c)}.
$$
It is straightforward to verify that for $s>0$, the triples of lengths
$$
(s,s,s), \quad 
\left(\sqrt{2}s, s, s\right), \quad
\left(s, \sqrt{2}s, s\right), \quad \text{ and }\quad 
\left(s, s, \sqrt{2}s\right),
$$
satisfy the three constraints and
correspond to triangles edge lengths since they are positive and satisfy the triangle inequalities. While the first corresponds to an equilateral triangle, the other three represent the three possible assignments of edge lengths in a right-angled triangle.
In these cases, $\bX_{3}$ and $\bX_{63}$ coincide with the midpoint of the hypothenuse. Therefore, the triangle centers are not essentially different.
\end{example}

\begin{proof}[Proof of \lemref{lem:essentiallDifferent}]

To test whether two traceable triangle centers are essentially different, it is sufficient to compare their normalized triangle center functions. For two triangle center functions $\psi_1$ and $\psi_2$, we consider the system of equations 
$$
\frac{\psi_1(a,b,c)}{\nmlz{\psi_1}(a,b,c)} = \frac{\psi_2(a,b,c)}{\nmlz{\psi_2}(a,b,c)},
\quad
\frac{\psi_1(b,c,a)}{\nmlz{\psi_1}(a,b,c)} = \frac{\psi_2(b,c,a)}{\nmlz{\psi_2}(a,b,c)},
\quad \text{and} \quad 
\frac{\psi_1(c,a,b)}{\nmlz{\psi_1}(a,b,c)} = \frac{\psi_2(c,a,b)}{\nmlz{\psi_2}(a,b,c)}.
$$
Since this system of equations has rank two, we solve it for two of its unknowns, the side lengths $b$ and $c$.

The attached \emph{Mathematica} notebook verifies that the only solutions obtained either correspond to the equilateral triangle or no triangle.
\end{proof}

\begin{proof}[Proof of \lemref{lem:cycaffineCenter}]

We show that the associated centers take the claimed form.
First, note that
\begin{align*}
\nmlz{\psi_{\bomega}} &= a\, \psi_{\bomega}(a,b,c)  + b\, \psi_{\bomega}(b,c,a) + c\, \psi_{\bomega}(c,a,b) \\
&= a\, \omega_0(a,b,c)\, \psi_0(a,b,c) + b\, \omega_0(b,c,a)\,\psi_0(b,c,a)  + c\, \omega_0(c,a,b)\,\psi_0(c,a,b) \\
  & \quad +\, a\,\omega_1(a,b,c) \,\psi_0(a,b,c)  + b\, \omega_1(b,c,a) \,\psi_0(b,c,a)  + c\,\omega_1(c,a,b)\, \psi_0(c,a,b)\\
&=\omega_0\left( a\, \psi_0(a,b,c) + b\, \psi_0(b,c,a)  + c\, \psi_0(c,a,b) \right) \\
& \quad  +\, \omega_1\left( a\, \psi_0(a,b,c)  + b\, \psi_0(b,c,a)  + c\, \psi_0(c,a,b) \right)\\
  &=\omega_0 \nmlz{\psi_0}  + \omega_1 \nmlz{\psi_1}.
\end{align*}

If \(
 \nmlz{\psi_{\bomega}}\neq 0\), then by the definition of the associated triangle center, a similar computation yields 
\begin{align*}
\bX_{\psi_{\bomega}} &= \frac{1}{\nmlz{\psi_{\bomega}}}\left(
a\, \psi_{\bomega}(a,b,c) \bA + b\, \psi_{\bomega}(b,c,a) \bB + c\, \psi_{\bomega}(c,a,b) \bC \right)\\
  &= \frac{1}{\nmlz{\psi_{\bomega}}}\left( \omega_0\nmlz{\psi_0} \bX_{\psi_0} + \omega_1 \nmlz{\psi_1} \bX_{\psi_1}\right).
\end{align*}
 This verifies the identity stated in~\teqref{eqn:x_omega}, and since coefficients sum up to 1, collinearity follows.

  Furthermore, $\psi_{\bomega}$ being a cyclic-affine combination of $\psi_{0}$ and $\psi_{1}$ (see~\teqref{eqn:cyclicaffine}), implies that we can express $\psi_1$ in terms of $\psi_0$ and $\psi_{\bomega}$ as 
    $$
     \psi_{1}(a,b,c) = -\frac{\omega_0(a,b,c)}{\omega_1(a,b,c)} \psi_0(a,b,c) +\frac{1}{ \omega_1(a,b,c)} \psi_{\bomega}(a,b,c).
    $$
    Note that $\deg(\nicefrac{\psi_{\bomega}}{\omega_1})= \deg(\nicefrac{\omega_0 \psi_0}{\omega_1})$ and that the coefficients of the triangle center functions are still cyclic and nowhere zero. Consequently, $\psi_1\in \spann_{\cycT}(\psi_0, \psi_1)$, which proves the claim.     
\end{proof}

\begin{proof}[Proof of \lemref{lem:cycAffineIndependence}]
     To show that the definition of cyclic-affine combinations is independent of the choice of triangle center functions, assume that $\psi_{\bomega}$ is the triangle center function in \teqref{eqn:cyclicaffine} and $\psi_1 \cong \tilde\psi_1$, \ie, $\tilde \psi_1 = f \psi_1$ for $f\in\cycT$.
    Then,
    $$
    \psi_{\bomega}(a,b,c) = \omega_0(a,b,c) \psi_0(a,b,c) + \frac{\omega_1(a,b,c)}{f(a,b,c)} \tilde \psi_1(a,b,c),
    $$
    with $\tilde\omega_1 = \nicefrac{\omega_1}{f}\in \cycT$ and $\deg(\omega_0 \psi_0) = \deg(\tilde\omega_1 \tilde\psi_1)$, showing that 
    $$
    \psi_{\bomega} \in \spann_{\cycT}(\psi_0, \psi_1) \quad \Longleftrightarrow
    \quad 
     \psi_{\bomega} \in \spann_{\cycT}(\psi_0, \tilde\psi_1).
    $$
    Consequently, $\spann_{\cycT}(\psi_0, f_1 \psi_1) = \spann_{cycT}(\psi_0, \psi_1)$.

    An analogous argument can be made for a cyclic factor of $\psi_0$.
 \end{proof}

\begin{example}[Collinearity of triangle centers $\bX_3$, $\bX_6$, and $\bX_{15}$]\label{ex:X15inBrocard}
\lemref{lem:cycAffineIndependence} allows us to algebraically recover known collinearities, such as the collinearity of the triangle centers $\bX_3$, $\bX_6$, and $\bX_{15}$, corresponding to the triangle center functions listed in \tblref{tab:trianglecenters}.

As discussed in \exref{ex:traceableTriangleCenters}, the triangle center function $\phi_{15}$ in \tblref{tab:trianglecenters} can be rewritten using only the triangle's edge lengths, see~\teqref{eqn:phi15abc}.  
By applying~\lemref{lem:equivCenterFunctions}, we obtain an equivalent representation of the triangle center function $\phi_{15}$ by multiplying by the factor $2 a b c \in \cycT$, 
\begin{equation}\label{eqn:phi15}
\phi_{15}(a,b,c) \cong a \left( \sqrt{3} (-a^2 + b^2 + c^2) + 4 \area\right).
\end{equation}

We observe that this expression allows for a natural decomposition into two components, 
$$
\phi_{15}(a,b,c) = \omega_0(a,b,c)\, \phi_3(a,b,c) +  \omega_1(a,b,c) \,\phi_6(a,b,c), 
$$
where $\omega_0(a,b,c) = \sqrt{3} \in \cycT$ and $\omega_1(a,b,c) = 4 \area \in \cycT$ with $\deg(\omega_0 \phi_3) = \deg(\omega_1\phi_6) = 3$. 
Furthermore, using the simplified expressions for the traces in \teqref{eqn:trace36}
 and \teqref{eqn:x_omega}, we obtain 
\begin{align*}
\bX_{15} 
&= \frac{4\sqrt{3} \area}{4\sqrt{3} \area + a^2 + b^2 + c^2}\bX_3 + \frac{a^2 + b^2 + c^2}{4\sqrt{3} \area + a^2 + b^2 + c^2}\bX_6.
\end{align*} 
Note how the coefficients of $\bX_{15}$ depend on the shape of the triangle.
\end{example}

\begin{proof}[Proof of \lemref{lem:triangleLineCycAffine}]
Let $\bX_{\psi}$ be a traceable triangle center collinear with $\bX_{\psi_0}$ and $\bX_{\psi_1}$ that is essentially different from $\bX_{\psi_0}$ and $\bX_{\psi_1}$. It follows that there exists a function $f\colon\triangle\to\RR$ with $f(a,b,c) \neq 0$ for all $\Delta\in\triangle$, such that
$$
\bX_{\psi} = \left(1-f(a,b,c)\right) \bX_{\psi_0} + f(a,b,c) \bX_{\psi_1}.
$$

We start by showing that $\bX_{\psi}$ being a triangle center implies $f(a,b,c)\in\cycT$, \ie, it is cyclic, homogeneous (in this case with homogeneity degree 0), and bi-symmetric.
To this end, let $\psi_0$, $\psi_1$ be the triangle center functions corresponding to $\bX_{\psi_0}$ and $\bX_{\psi_1}$. It follows that   
\begin{align*}
\bX_{\psi} &=  \left(1-f(a,b,c)\right)\frac{a\, \psi_0(a,b,c) \bA + b \,\psi_0(b,c,a)\bB + c\, \psi_0(c,a,b)\bC}{\nmlz{\psi_0}(a,b,c)}\\ 
& \quad +  f(a,b,c) \frac{a\, \psi_1(a,b,c) \bA + b\, \psi_1(b,c,a)\bB + c \,\psi_1(c,a,b)\bC}{\nmlz{\psi_1}(a,b,c)} \\
&=  a \left( \left(1-f(a,b,c)\right) \frac{\psi_0(a,b,c)}{\nmlz{\psi_0}(a,b,c)} + f(a,b,c) \frac{\psi_1(a,b,c)}{\nmlz{\psi_1}(a,b,c)}\right) \bA \\
& \quad + b \left(\left(1-f(a,b,c)\right) \frac{\psi_0(b,c,a)}{\nmlz{\psi_0}(a,b,c)} + 
f(a,b,c) \frac{\psi_1(b,c,a)}{\nmlz{\psi_1}(a,b,c)}\right) \bB \\
& \quad + c \left( \left(1-f(a,b,c)\right) \frac{\psi_0(c,a,b)}{\nmlz{\psi_0}(a,b,c)} + 
f(a,b,c) \frac{\psi_1(c,a,b)}{\nmlz{\psi_1}(a,b,c)}\right) \bC.
\end{align*}
Since $\bX_{\psi}$ is a triangle center, it relates to a triangle center function $\psi$ as follows
$$
\bX_{\psi} = \frac{1}{\nmlz{\psi}(a,b,c)}\left(a\,  \psi(a,b,c) \bA + b \,\psi(b,c,a)\bB + c \,\psi(c,a,b)\bC\right).
$$
Comparison of the coefficients, we conclude that 
\begin{align*} 
\frac{\psi(a,b,c)}{\nmlz{\psi}(a,b,c)} &= \left(1-f(a,b,c)\right) \frac{\psi_0(a,b,c)}{\nmlz{\psi_0}(a,b,c)} + f(a,b,c) \frac{\psi_1(a,b,c)}{\nmlz{\psi_1}(a,b,c)} \\
\frac{\psi(b,c,a)}{\nmlz{\psi}(a,b,c)} &= \left(1-f(a,b,c)\right) \frac{\psi_0(b,c,a)}{\nmlz{\psi_0}(a,b,c)} + f(a,b,c) \frac{\psi_1(b,c,a)}{\nmlz{\psi_1}(a,b,c)} \\
\frac{\psi(c,a,b)}{\nmlz{\psi}(a,b,c)} &= \left(1-f(a,b,c)\right) \frac{\psi_0(c,a,b)}{\nmlz{\psi_0}(a,b,c)} + f(a,b,c) \frac{\psi_1(c,a,b)}{\nmlz{\psi_1}(a,b,c)}.
\end{align*}

\begin{itemize}
    \item \emph{Cyclicity:}
We find the solutions $f_i$ that solve the $i$-th equation for $f$, namely
\begin{align*}
    f_0(a,b,c) &= \left(\frac{\psi_0(a,b,c)}{\nmlz{\psi_0}(a,b,c)}
    - \frac{\psi(a,b,c)}{\nmlz{\psi}(a,b,c)} 
    \right)
    \left(\frac{\psi_0(a,b,c)}{\nmlz{\psi_0}(a,b,c)} - \frac{\psi_1(a,b,c)}{\nmlz{\psi_1}(a,b,c)}\right)^{-1} \\
    f_1(a,b,c) &= \left(\frac{\psi_0(b,c,a)}{\nmlz{\psi_0}(a,b,c)}
    - \frac{\psi(b,c,a)}{\nmlz{\psi}(a,b,c)} 
    \right)
    \left(\frac{\psi_0(b,c,a)}{\nmlz{\psi_0}(a,b,c)} - \frac{\psi_1(b,c,a)}{\nmlz{\psi_1}(a,b,c)}\right)^{-1} \\
    f_2(a,b,c) &= \left(\frac{\psi_0(c,a,b)}{\nmlz{\psi_0}(a,b,c)}
    - \frac{\psi(c,a,b)}{\nmlz{\psi}(a,b,c)} 
    \right)
    \left(\frac{\psi_0(c,a,b)}{\nmlz{\psi_0}(a,b,c)} - \frac{\psi_1(c,a,b)}{\nmlz{\psi_1}(a,b,c)}\right)^{-1}.
\end{align*}
Note that $f_0(c,a,b) = f_1(b,c,a) = f_2(a,b,c)$. Consequently, the unique solution $f(a,b,c) := f_0(a,b,c)$ is cyclic. 

\item \emph{Homogeneity:} Next, using the homogeneity property of the triangle center functions $\psi$, $\psi_0$, and $\psi_1$, we show that $f(ta,tb,tc) = f(a,b,c)$, that is,
\begin{align*}
f(t a,t b,tc) &= \left(\frac{\psi_0(ta,tb,tc)}{\nmlz{\psi_0}(ta,tb,tc)}
    - \frac{\psi(ta,tb,tc)}{\nmlz{\psi}(ta,tb,tc)} 
    \right)
    \left(\frac{\psi_0(ta,tb,tc)}{\nmlz{\psi_0}(ta,tb,tc)} - \frac{\psi_1(ta,tb,tc)}{\nmlz{\psi_1}(ta,tb,tc)}\right)^{-1}\\
    &= \left(\frac{1}{t}\frac{\psi_0(a,b,c)}{\nmlz{\psi_0}(a,b,c)}
    - \frac{1}{t}\frac{\psi(a,b,c)}{\nmlz{\psi}(a,b,c)} 
    \right)
    \left(\frac{1}{t}\frac{\psi_0(a,b,c)}{\nmlz{\psi_0}(a,b,c)} - \frac{1}{t}\frac{\psi_1(a,b,c)}{\nmlz{\psi_1}(a,b,c)}\right)^{-1}\\
&= f(a,b,c).
\end{align*}

\item \emph{Bi-symmetry: } Building on the bi-symmetry of the triangle center functions $\psi$, $\psi_0$, and $\psi_1$, a similar argument shows that 
\begin{align*}
f(a,c,b) &= \left(\frac{\psi_0(a,c,b)}{\nmlz{\psi_0}(a,b,c)}
    - \frac{\psi(a,c,b)}{\nmlz{\psi}(a,b,c)} 
    \right)
    \left(\frac{\psi_0(a,c,b)}{\nmlz{\psi_0}(a,b,c)} - \frac{\psi_1(a,c,b)}{\nmlz{\psi_1}(a,b,c)}\right)^{-1} \\
    &= \left(\frac{\psi_0(a,b,c)}{\nmlz{\psi_0}(a,b,c)}
    - \frac{\psi(a,b,c)}{\nmlz{\psi}(a,b,c)} 
    \right)
    \left(\frac{\psi_0(a,b,c)}{\nmlz{\psi_0}(a,b,c)} - \frac{\psi_1(a,b,c)}{\nmlz{\psi_1}(a,b,c)}\right)^{-1}\\
&= f(a,b,c).
\end{align*}
\end{itemize}
We therefore conclude that the triangle center function of $\bX_{\psi}$ allows the representation 
$$
\psi(a,b,c) = \left(1 - f(a,b,c)\right) \frac{\psi_0(a,b,c)}{\nmlz{\psi_0}(a,b,c)} + f(a,b,c) \frac{\psi_1(a,b,c)}{\nmlz{\psi_1}(a,b,c)},
$$
where $f\in\cycT$ with homogeneity degree 0.

Finally, we show that the above expression is equivalent a cyclic-affine combination of $\psi_0$ and $\psi_1$, see \teqref{eqn:cyclicaffine}. To this end, we rewrite 
\begin{align*}
\psi_{\omega}(a,b,c) &=
 \left(\omega_0 \nmlz{\psi_0} + \omega_1 \nmlz{\psi_1}\right) \left(\left(1 - \frac{\omega_1\nmlz{\psi_1}}{\omega_0 \nmlz{\psi_0} + \omega_1 \nmlz{\psi_1}}\right) \psi_0 +   \frac{\omega_1\nmlz{\psi_1}}{\omega_0 \nmlz{\psi_0} + \omega_1 \nmlz{\psi_1}} \psi_1\right).
\end{align*}
Defining the functions 
\begin{align*} 
\omega = \omega_0 \nmlz{\psi_0} + \omega_1 \nmlz{\psi_1} \quad 
\text{and} \quad 
f = \frac{\omega_1 \nmlz{\psi_1}}{\omega_0 \nmlz{\psi_0} + \omega_1 \nmlz{\psi_1}},
\end{align*} 
we obtain the conversion between the equivalent triangle function representations $\psi_{\omega} \cong \omega\, \psi$, as 
$$
\psi_{\omega} = \omega \left(\left(1-f\right)\frac{\psi_0}{\nmlz{\psi_0}} + f \frac{\psi_1}{\nmlz{\psi_1}}\right) \cong (1-f)\frac{\psi_0}{\nmlz{\psi_0}} + f \frac{\psi_1}{\nmlz{\psi_1}} = \psi,
$$
which concludes the proof.
\end{proof}

\begin{proof}[Proof of \lemref{lem:essDifferentConstAffine}]
We first show that $\psi_1$ is essentially different from $\psi_{\lambda_0:\lambda_1}$ by a proof by contradiction. 
For the sake of the argument, we assume that $\psi_0$ and $\psi_{\lambda_0:\lambda_1}$ are not essentially different. Consequently, there exists a triangle $(a,b,c)$, which is not equilateral, for which 
\begin{align*}
\frac{\psi_0(a,b,c)}{\nmlz{\psi_0}(a,b,c)} &= 
\frac{\psi_{\lambda_0:\lambda_1}(a,b,c)}{\nmlz{\psi_{\lambda_0:\lambda_1}}(a,b,c)}\\
&= \frac{\lambda_0\, \psi_0(a,b,c) + \lambda_1 \, \psi_1(a,b,c)}{\lambda_0 \nmlz{\psi_0}(a,b,c) + \lambda_1 \nmlz{\psi_1}(a,b,c)}.
\end{align*}
Simplifying this equation results in 
    $$
   \frac{\psi_0(a,b,c)}{\nmlz{\psi_0}(a,b,c)} = \frac{\psi_1(a,b,c)}{\nmlz{\psi_1}(a,b,c)},
$$
which is a contradiction.  Consequently, $\psi_0$ and $\psi_{\lambda_0:\lambda_1}$ are essentially different. 
The argumentation for the case where $\psi_{\lambda_0:\lambda_1}$ and $\psi_1$ are essentially different is analogous. 

To show that $\psi_{\lambda_0:\lambda_1}$ and $\psi_{\overline{\lambda}_0:\overline{\lambda}_1}$ are essentially different, we again employ a proof by contradiction. 
Again, assume that $\psi_{\lambda_0:\lambda_1}$ and $\psi_{\overline{\lambda}_0:\overline{\lambda}_1}$ are not essentially different. Consequently, there exists a non-equilateral triangle $(a,b,c)$ for which 
\begin{align*}
    \frac{\psi_{\lambda_0:\lambda_1}(a,b,c)}{\nmlz{\psi_{\lambda_0:\lambda_1}}(a,b,c)} &= \frac{\psi_{\overline{\lambda}_0:\overline{\lambda}_1}(a,b,c)}{\nmlz{\psi_{\overline{\lambda}_0:\overline{\lambda}_1}}(a,b,c)}.
\end{align*}
Similar to before, this equation simplifies to 
$$
0 = \left(\frac{\lambda_0}{\lambda_1}- \frac{\overline{\lambda}_0}{\overline{\lambda}_1}\right)\left(\frac{\psi_0(a,b,c)}{\nmlz{\psi_0}(a,b,c)} - \frac{\psi_1(a,b,c)}{\nmlz{\psi_1}(a,b,c)}\right).
$$
Since the first factor is non-zero, the second factor needs to vanish, which is a contradiction to our assumption that $\psi_0$ and $\psi_1$ are essentially different. Consequently, $\psi_{\lambda_0:\lambda_1}$ and 
$\psi_{\overline{\lambda}_0:\overline{\lambda}_1}$ are essentially different.
\end{proof}

\begin{lemma}\label{lem:AffineFromCyclic}
Two essentially different triangle centers $\bX_{\psi_0}$ and $\bX_{\psi_1}$, together with a third triangle center $\bX_{\psi_{\bomega}} \in \spann_{\cycT}(\bX_{\psi_0}, \bX_{\psi_1})$ define a family of constant-affine triangle centers that is independent of the choice of triangle center functions and contains $\bX_{\psi_{\bomega}}$. We refer to this family as $\spann_{\const}(\bX_{\psi_0}, \bX_{\psi_1}, \bX_{\psi_{\bomega}})$. 

Although by construction $\bX_{\psi_{\bomega}}\in \spann_{\const}(\bX_{\psi_0}, \bX_{\psi_1}, \bX_{\psi_{\bomega}})
$, the other two points are by definition never contained in the family, \ie,  $\bX_{\psi_0}, \bX_{\psi_1}\not\in\spann_{\const}(\bX_{\psi_0}, \bX_{\psi_1}, \bX_{\psi_{\bomega}})$.
However, 
$$
\bX_{\psi_{\overline{\bomega}}}\in \spann_{\const}(\bX_{\psi_0}, \bX_{\psi_1}, \bX_{\psi_{\bomega}})
\quad \text{implies}\quad
\bX_{\psi_{\bomega}}\in \spann_{\const}(\bX_{\psi_0}, \bX_{\psi_1}, \bX_{\psi_{\overline{\bomega}}}).
$$
\end{lemma}

\begin{proof}[Proof of \lemref{lem:AffineFromCyclic}]
Let $\bX_{\psi_0}$ and $\bX_{\psi_1}$ be two essentially different triangle centers and $\bX_{\psi_{\bomega}} \in \spann_{\cycT}(\bX_{\psi_0}, \bX_{\psi_1})$.
It follows from the definition of cyclic-affine triangle center combinations that 
    $\psi_{\bomega}$ allows a decomposition according to 
    \teqref{eqn:cyclicaffine}
    with $\omega_0, \omega_1\in\cycT$.    
    Note that the decomposition is unique for non-equilateral triangle configurations, with the coefficients given by
    \begin{align*}
    \omega_0(a,b,c) &= \frac{\psi_1(a,b,c) \psi_2(b,c,a) - \psi_1(b,c,a) \psi_2(a,b,c)}{\psi_0(b,c,a) \psi_1(a,b,c) - \psi_0(a,b,c) \psi_1(b,c,a)}, \\
    \omega_1(a,b,c) &= \frac{\psi_0(a,b,c) \psi_2(b,c,a) - \psi_0(b,c,a) \psi_2(a,b,c)}{\psi_0(b,c,a) \psi_1(a,b,c) - \psi_0(a,b,c) \psi_1(b,c,a)}.
    \end{align*}
    This is a direct consequence of the fact that for $\triangle \neq (a,a,a)$, the matrix $\bM_{\psi}$ is of rank two (its rank cannot be one since the triangle center functions $\psi_0$ and $\psi_1$ correspond to essentially different triangle centers).
    Therefore, we obtain a set of constant-affine triangle center functions 
    $$
    \psi_{\gamma_0:\gamma_1}(a,b,c) = \gamma_0\, \omega_0(a,b,c)\, \psi_0(a,b,c) + \gamma_1 \,\omega_1(a,b,c) \,\psi_1(a,b,c)
    $$
    that contain $\psi_{\bomega}$ for $\gamma_0 = \gamma_1 = 1$. 
\end{proof}

\begin{example}[Structure of triangle center families on the Brocard axis and Kiepert hyperbola]\label{ex:brocard}
Recall that the line spanned by the circumcenter $\bX_3$ and the symmedian point $\bX_6$ is known as the \emph{Brocard axis} and its isogonal conjugate, the \emph{Kiepert hyperbola}, is a circumconic, that is, a conic section passing through the vertices of the underlying triangle~\cite[page 235]{Kimberling:1998:TCC}. 

\begin{figure}[t]
\centering
\begin{footnotesize}
\begin{overpic}[width=0.5\textwidth]{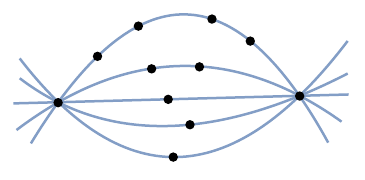}
\put(14,15.5){$\bX_3$}
\put(80,17){$\bX_6$}
\put(21,36){$\bX_{15}$}
\put(33,44.7){$\bX_{16}$}
\put(59,45.5){$\bX_{61}$}
\put(70,39){$\bX_{62}$}
\put(39,33){$\bX_{32}$}
\put(52,33.5){$\bX_{39}$}
\put(43,24.5){$\bX_{50}$}
\put(50,17.5){$\bX_{52}$}
\put(45,9){$\bX_{58}$}
\end{overpic}%
\end{footnotesize}%
\caption{Schematic illustration of the relationship between constant-affine combinations of triangle centers $\bX_3$ and $\bX_6$ showing the considered triangle centers. }\label{fig:constAffineBrocard}
\end{figure}
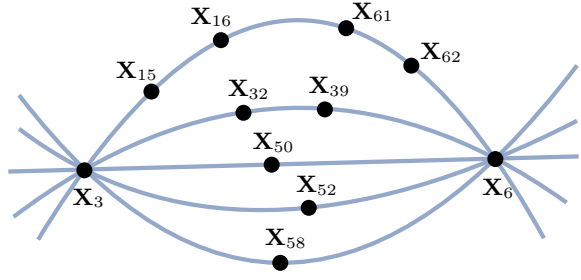

In this example we will illustrate the concepts of cyclic-affine and constant-affine triangle center families by studying the structure of sets of triangle centers on the Brocard axis. In our main result, stated in \thmref{thm:main} on triangle curves, two of the constant-affine families associated with the Brocard axis introduced here play a crucial role.

\begin{table}[!b]
\centering
    \begin{tabular}{ccc}
        $\bX_i$ & $\phi_i$ (ETC) & $\phi_i$ (mod.)  
        \\ 
        \hline
        $\bX_3$ & $a(-a^2+b^2+c^2)$ & $a(-a^2+b^2+c^2)$ 
        \\
        $\bX_6$ & $a$ & $a$ 
        \\
        $\bX_{15}$ & 
        $\sin\left(\alpha + \frac{\pi}{3}\right)$ & 
        $a\left(\sqrt{3}(-a^2+b^2+c^2) + 4 \area\right)$ 
        \\
        $\bX_{16}$ & 
        $\sin\left(\alpha - \frac{\pi}{3}\right)$ & 
        $ a\left(\sqrt{3}(-a^2+b^2+c^2) - 4 \area\right)$ 
        \\
        $\bX_{32}$ & 
        $a^3$ & 
        $2 a^3$ 
        \\
        $\bX_{39}$ & 
        $a(b^2+c^2)$ & 
        $2a(b^2+c^2)$ 
        \\
        $\bX_{50}$ & 
        $\sin\left(3\alpha\right)$ & 
        $2 a^3\left(a^4 + b^4 + c^4 -2a^2(b^2+c^2) + b^2c^2\right)$ 
        \\
        $\bX_{52}$ & 
        $\sec \alpha\left(\sec(2\beta) + \sec(2\gamma)\right)$ & 
        $2a\left(a^4 + b^4+c^4- 2a^2(b^2+c^2)\right)
        \left((b^2-c^2)^2 - a^2(b^2+c^2)\right)$ 
        \\
        $\bX_{58}$ & 
        $\frac{a}{b+c}$ & 
        $2a (a+b)(a+c)$ 
       \\
        $\bX_{61}$ & 
        $\sin\left(\alpha + \frac{\pi}{6}\right)$ & 
        $a \left(-a^2+b^2+c^2 + 4 \sqrt{3}  \area \right)$ 
        \\
        $\bX_{62}$ & 
        $\sin\left(\alpha - \frac{\pi}{6}\right)$ & 
        $a \left(-a^2+b^2+c^2 - 4 \sqrt{3}  \area \right)$ 
    \end{tabular}\caption{Table depicting the first one-hundred triangle centers that coincide with the Brocard axis together with their triangle center function as stated in the ETC, and their simplification using the law of cosines.
}\label{tab:brocardaxis}
\end{table}

Among the first one hundred triangle centers listed in Kimberling~\cite{Kimberling:2025:ETC}, in addition to $\bX_3$ and $\bX_6$, further points on the Brocard axis include the first isodynamic point $\bX_{15}$, the second isodynamic point $\bX_{16}$, the third power point $\bX_{32}$, the Brocard midpoint $\bX_{39}$, 
$\bX_{50}$, the orthocenter of the orthic triangle $\bX_{52}$, 
$\bX_{58}$,
$\bX_{61}$, as well as $\bX_{62}$. \tblref{tab:brocardaxis} lists their triangle center functions, along with equivalent expressions without trigonomic functions, used for the subsequent computations. 

\begin{table}[t]
    \centering
    \begin{tabular}{ccc}
        $\phi_i$ (mod.)& $\omega_{i,3}$ & $\omega_{i,6}$ \\
        \hline
        $\phi_{15}$ & $\sqrt{3}$ & $4\area$ \\
        $\phi_{16}$ & $\sqrt{3}$ & -$4\area$ \\
        $\phi_{32}$ & $-1$ & $a^2+b^2+c^2$ \\
        $\phi_{39}$ & $1 $ & $a^2+b^2+c^2$ \\
        $\phi_{50}$ & $(4\area)^2$ & 
        $-2a^2b^2c^2-(-a^2+b^2+c^2)(a^2-b^2+c^2)(a^2+b^2-c^2)$ \\
        $\phi_{52}$ & $\omega_{50,6}$ & $(4\area)^4$ \\
        $\phi_{58}$ & $-1$ & $(a + b + c)^2$ \\
        $\phi_{61}$ & $1$ & $4 \sqrt{3} \area$ \\
        $\phi_{62}$ & $1$ & $-4\sqrt{3}\area$ 
    \end{tabular}
    \caption{Table listing the cyclic-affine coefficients of the first one-hundred triangle centers that coincide with the Brocard axis. Specifically,  $\phi_i = \omega_{i,3} \phi_3+ \omega_{i,6} \phi_6$, 
    where $\phi_i$ is the expression in the third column of \tblref{tab:brocardaxis}, which is equivalent to the expression given in the ETC listed in the second column. }
    \label{tab:coeff_omega_combination}
\end{table}

First, we confirm that indeed all the triangle centers listed are indeed cyclic-affine combinations of $\bX_3$ and $\bX_6$, that is 
$$
B := \{\bX_{{15}}, \bX_{{16}}, \bX_{{32}}, \bX_{{39}}, \bX_{{52}}, \bX_{{52}}, \bX_{{58}}, \bX_{{61}}, \bX_{{62}} \}  \subset \spann_{\cycT}(\bX_{{3}}, \bX_{6}),
$$
To this end, we find coefficients $\omega_0, \omega_1 \in \cycT$ such that
$$
\phi_i(a,b,c) = \omega_{i,3}(a,b,c)\, \phi_3(a,b,c) + \omega_{i,6}(a,b,c)\,\phi_6(a,b,c),
$$
where $\phi_i(a,b,c)$ are the triangle center functions corresponding to the centers in $B$. The resulting
coefficients are stated in \tblref{tab:coeff_omega_combination}.

Any of the triangle centers in $B$, together with $\bX_3$ and $\bX_6$, span a constant-affine family of triangle centers that coincide with the Brocard axis. We show how these families relate, in particular, we confirm that 
\begin{align*}
\spann_{\const}(\bX_{3}, \bX_{6}, \bX_{{15}}) &= \spann_{\const}(\bX_{3}, \bX_{6}, \bX_{{16}}) \\
&= \spann_{\const}(\bX_{3}, \bX_{6}, \bX_{61}) \\
&= \spann_{\const}(\bX_{3}, \bX_{6}, \bX_{{62}}) 
\end{align*}
and 
$$
\spann_{\const}(\bX_{3}, \bX_{6}, \bX_{{32}}) = \spann_{\const}(\bX_{3}, \bX_{6}, \bX_{{39}}), 
$$
as depicted in \figref{fig:constAffineBrocard}.

For the algebraic confirmation of these statements, for every triangle center $\bX_i$ in $B$, we parametrize the triangle center functions in the corresponding constant-affine triangle center family using two scalars $\lambda_0,\lambda_1\in\RR\backslash\{0\}$ by 
$$
 \psi_{i;\lambda_0:\lambda_1}(a,b,c) = \lambda_0 \,  \omega_{i,3}(a,b,c)\, \phi_3(a,b,c) + \lambda_1 \,  \omega_{i,6}(a,b,c)\, \phi_6(a,b,c), 
$$
where $\omega_{i,3}$ and $\omega_{i,6}$ are the corresponding coefficients in~\tblref{tab:coeff_omega_combination} that ensure appropriate normalization. Note that for all considered triangle centers, we have that 
$$
\phi_i(a,b,c) = \psi_{i,1:1}(a,b,c).
$$

If two triangle centers $\bX_i$ and $\bX_j$ in $B$ belong to the same constant-affine family, we have
$$
\phi_j(a,b,c) \cong \psi_{i;\lambda_0:\lambda_1}(a,b,c), \quad \text{or equivalently, } \quad 
\phi_i(a,b,c) \cong \psi_{j;\tilde{\lambda_0}:\tilde{\lambda_1}}(a,b,c), 
$$
for some constants $\lambda_0$, $\lambda_1$, $\tilde{\lambda_0}$, and $\tilde{\lambda_1}$.
The  relationships found by Mathematica are summarized in~\tblref{tab:constantAffineBrocard}. 

In the main result of this paper, presented in~\thmref{thm:main}, we show that the triangle centers in $\spann_{\const}(\bX_3, \bX_6, \bX_{15})$ and $\spann_{\const}(\bX_3, \bX_6, \bX_{32})$ are semi-invariant curve-generating centers.

\begin{table}[t]
\centering
\begin{tabular}{c|cccc}
    $\spann_{\const}(\bX_3, \bX_6, \bX_{15})$ & 
    $\psi_{15,(-1):1} \cong \phi_{16}$ &
    $\psi_{15,1:3} \cong \phi_{61} $ & 
    $\psi_{15,(-1):3} \cong \phi_{62}$ \\
   $\spann_{\const}(\bX_3, \bX_6, \bX_{16})$ & 
    $\psi_{16,(-1):1} \cong \phi_{15}$ & 
    $\psi_{16,(-1):3} \cong \phi_{61} $ & 
    $\psi_{16,1:3} \cong \phi_{62}$ \\
     $\spann_{\const}(\bX_3, \bX_6, \bX_{32})$ & 
    $\psi_{32,(-1):1} \cong \phi_{39}$ & &
    \\
    $\spann_{\const}(\bX_3, \bX_6, \bX_{39})$ & 
    $\psi_{39,(-1):1} \cong \phi_{32}$ & &
    \\
    $\spann_{\const}(\bX_3, \bX_6, \bX_{61})$ & 
    $\psi_{61,3:1} \cong \phi_{15}$ &
    $\psi_{61,(-1):3} \cong \phi_{16} $ & 
    $\psi_{61,(-1):1} \cong \phi_{62}$ 
    \\
      $\spann_{\const}(\bX_3, \bX_6, \bX_{62})$ & 
    $\psi_{62,(-3):1} \cong \phi_{15}$ &
    $\psi_{62,1:3} \cong \phi_{16} $ & 
    $\psi_{62,(-1):1} \cong \phi_{61}$ 
    \end{tabular}
\caption{Algebraically found non-trivial identities relating the triangle families $\psi_{i,\omega_0:\omega_1}$ and triangle centers on the Brocard axis.   }\label{tab:constantAffineBrocard}    
\end{table}

\end{example}

\section{Triangle Families}

\begin{proof}[Proof of~\lemref{lem:propertiesFamilies}]~

\begin{enumerate}
    \item
The vertices of $\Delta_t^{\bPsi}$ read 
\begin{align}\label{eqn:ABCPsi}
\bA_t^{\bPsi} & = \frac{1}{\Psi_1(t) + \Psi_2(t) + \Psi_3(t)} \left(\Psi_1(t)\bA + \Psi_2(t) \bB + \Psi_3(t) \bC\right), \nonumber \\
\bB_t^{\bPsi} & = \frac{1}{\Psi_1(t) + \Psi_2(t) + \Psi_3(t)} \left(\Psi_3(t)\bA + \Psi_1(t) \bB + \Psi_2(t) \bC\right), \\
\bC_t^{\bPsi} & = \frac{1}{\Psi_1(t) + \Psi_2(t) + \Psi_3(t)}  \left(\Psi_2(t)\bA + \Psi_3(t) \bB + \Psi_1(t) \bC\right). \nonumber
\end{align}
A straightforward computation verifies that the centroids of $\triangle_t^{\bPhi}$ and $\triangle$ coincide
\[
 \frac{1}{3}\left( \bA + \bB + \bC\right) =  \frac{1}{3} \left(\bA_t^{\bPsi} + \bB_t^{\bPsi} + \bC_t^{\bPsi}\right).
\]

\item Without loss of generality, assume that $\Delta$ is centered at the origin. Consequently,  its vertices are related by 
 $$
 \bA = \bR \cdot \bC, \quad 
 \bB = \bR \cdot \bA, \quad  \text{and} \quad 
 \bC = \bR \cdot \bB, 
$$ 
where $\bR$ denotes the matrix corresponding to a rotation about $\nicefrac{2\pi}{3}$.
It follows from the expressions in Equation~\eqref{eqn:ABCPsi} that 
 $$
 \bR \cdot \bA_t^{\bPsi} =\bC_t^{\bPsi}, \quad
\bR \cdot \bB_t^{\bPsi} =\bA_t^{\bPsi}, \quad \text{and} \quad
 \bR \cdot \bC_t^{\bPsi} =\bB_t^{\bPsi}, \quad 
$$
that is, the vertices of $\Delta_t$ are related by a rotation by $\nicefrac{2\pi}{3}$ too. Consequently, $\Delta_t$ is an equilateral triangle.
\end{enumerate}

\end{proof}

\begin{proof}[Proof of Definition \& Lemma~\ref{deflemma:conctrianglefamilies}]
    The concatenation \((\bPhi\circ\bPsi)(t)\) is given by the generating functions
\begin{align}\label{eqn:concatenation}
    (\bPsi\circ\tilde\bPsi)_1 &= \frac{ (\Psi_1\tilde\Psi_1 + \Psi_2\tilde\Psi_3 + \Psi_3\tilde\Psi_2)\bA + (\Psi_1\tilde\Psi_2 + \Psi_2\tilde\Psi_1 + \Psi_3\tilde\Psi_3)\bB + (\Psi_1\tilde\Psi_3 + \Psi_2\tilde\Psi_2 + \Psi_3\tilde\Psi_1)\bC }{(\Psi_1 + \Psi_2 + \Psi_3)(\tilde\Psi_1 + \tilde\Psi_2 + \tilde\Psi_3)} \notag \\
    (\bPsi\circ\tilde\bPsi)_2 &= \frac{ (\Psi_1\tilde\Psi_3 + \Psi_2\tilde\Psi_2 + \Psi_3\tilde\Psi_1)\bA +  (\Psi_1\tilde\Psi_1 + \Psi_2\tilde\Psi_3 + \Psi_3\tilde\Psi_2)\bB + (\Psi_1\tilde\Psi_2 + \Psi_2\tilde\Psi_1 + \Psi_3\tilde\Psi_3)\bC }{(\Psi_1 + \Psi_2 + \Psi_3)(\tilde\Psi_1 + \tilde\Psi_2 + \tilde\Psi_3)} \\
    (\bPsi\circ\tilde\bPsi)_3 &= \frac{ (\Psi_1\tilde\Psi_2 + \Psi_2\tilde\Psi_1 + \Psi_3\tilde\Psi_3)\bA +  (\Psi_1\tilde\Psi_3 + \Psi_2\tilde\Psi_2 + \Psi_3\tilde\Psi_1)\bB + (\Psi_1\tilde\Psi_1 + \Psi_2\tilde\Psi_3 + \Psi_3\tilde\Psi_2)\bC }{(\Psi_1 + \Psi_2 + \Psi_3)(\tilde\Psi_1 + \tilde\Psi_2 + \tilde\Psi_3)}.  \notag
\end{align}

Furthermore note that 
$$
(\bPsi\circ\tilde\bPsi)_1+(\bPsi\circ\tilde\bPsi)_2+(\bPsi\circ\tilde\bPsi)_3 = \left(\Psi_1 + \Psi_2 + \Psi_3\right)\left(\tilde\Psi_1 + \tilde\Psi_2 + \tilde\Psi_3\right)\neq 0,
$$
and 
$$
(\bPsi\circ\tilde\bPsi)_1 \equiv (\bPsi\circ\tilde\bPsi)_2 \equiv (\bPsi\circ\tilde\bPsi)_3 
$$
implies 
$$
\bPsi_1 \equiv \bPsi_2 \equiv \bPsi_3 \quad \text{ and } \quad \tilde\bPsi_1 \equiv \tilde\bPsi_2 \equiv \tilde\bPsi_3,
$$
and thus the statement follows.
\end{proof}

\begin{proof}[Proof of \lemref{lem:proptriangleConcatenation}]~
\begin{enumerate}
    \item The communtativity of the concatenation of $\bPsi$ triangle families follows directly from Definition \& Lemma~\ref{deflemma:conctrianglefamilies}.
    \item To show that $\bPsi$-triangle families are associative, we first compute 
\begin{align*}
        \ttPsi \circ (\tPsi \circ \Psi) & = 
        \begin{pmatrix}
\Psi_3 \ttPsi_3\!+\!\Psi_2 \ttPsi_1\!+\! \Psi_1 \ttPsi_2   
 &  \Psi_1 \ttPsi_1  \!+\! \Psi_2 \ttPsi_3 \!+\! \Psi_3 \ttPsi_2  &
 \Psi_2 \ttPsi_2 \!+\!\Psi_1 \ttPsi_3 \!+\! \Psi_3 \ttPsi_1  \\
\Psi_2 \ttPsi_2 \!+\! \Psi_1 \ttPsi_3 \!+\! \Psi_3 \ttPsi_1  & \Psi_3 \ttPsi_3 \!+\!\Psi_1 \ttPsi_2 \!+\! \Psi_2 \ttPsi_1  & \Psi_1 \ttPsi_1 \!+\!\Psi_2 \ttPsi_3 \!+\! \Psi_3 \ttPsi_2 \\
\Psi_1 \ttPsi_1 \!+\! \Psi_2 \ttPsi_3 \!+\! \Psi_3 \ttPsi_2  & \Psi_2 \ttPsi_2 \!+\! \Psi_1 \ttPsi_3 \!+\! \Psi_3 \ttPsi_1  & \Psi_3 \ttPsi_3 \!+\!\Psi_1 \ttPsi_2 \!+\!\Psi_2 \ttPsi_1
\end{pmatrix}
\begin{pmatrix}
    \tPsi_1 \\ \tPsi_2\\ \tPsi_3
\end{pmatrix}
    \end{align*}
    
    Since the role of $\Psi$ and $\ttPsi$ are interchangable, 
    it follows from the above expression that 
    $$
    \ttPsi \circ (\tPsi \circ \Psi) = \Psi \circ (\tPsi \circ \ttPsi)
    $$
By commutativity of the concatenation, we have
$$
\Psi \circ (\tPsi \circ \ttPsi) = \Psi \circ (\ttPsi \circ \tPsi) = (\ttPsi \circ \tPsi) \circ  \Psi,
$$
and we ultimately conclude that the concatenation of triangle families is associative.

\item This follows directly form the definition of $\bPhiID \circ \bPsi$.

\item This claim follows readily from using Equation~\eqref{eqn:concatenation} and rephrasing 
    \begin{align*}
      \bA_t^{\bPsi^{-1}\circ\bPsi} =  \bA,  \quad 
    \bB_t^{\bPsi^{-1}\circ\bPsi} = \bB,
        \quad
        \text{ and } \quad
    \bC_t^{\bPsi^{-1}\circ\bPsi} =\bC
    \end{align*}
    as 
    $$
    \begin{pmatrix}
        \Psi_1 & \Psi_3 & \Psi_2 \\
        \Psi_2 & \Psi_1 & \Psi_3 \\
        \Psi_3 & \Psi_2 & \Psi_1
    \end{pmatrix}
    \begin{pmatrix}
        \Psi_1^{-1}\\
        \Psi_2^{-1}\\
        \Psi_3^{-1}
    \end{pmatrix}
    = 
    \begin{pmatrix}
        1\\
        0\\
        0
    \end{pmatrix}.
    $$
It follows from the definition of a triangle family that this matrix is invertible, as its determinant
    $$
    \Psi_1^3 + \Psi_2^3 + \Psi_3^3 - 3\Psi_1\Psi_2\Psi_3
    = \frac{1}{2}\left(\Psi_1 + \Psi_2+\Psi_3\right)
    \left((\Psi_1 - \Psi_2)^2+(\Psi_2-\Psi_3)^2 + (\Psi_3 - \Psi_1)^2\right)
   \neq 0.
    $$ 

\end{enumerate}
\end{proof}

\begin{proof}[Proof of Definition \& Lemma~\ref{lem:nedian}]
 Recall from \secref{sec:CyclicAffineTriangleCenterCombinations} that three points given by their barycentric coordinates $\bP = [P_1: P_2:P_3]^T$, $\bQ = [Q_1: Q_2: Q_3]^T$, and $\bS = [S_1: S_2: S_3]^T$ are collinear, if 
$$
0 = 
\det
\begin{pmatrix} 
\bP & \bQ & \bS
\end{pmatrix} = (\bP \times \bQ) \cdot \bS.
$$
Consequently, a line can be associated with a vector $\ell_{\bP\bQ} = \bP \times \bQ$ and incidence of a third point $\bS$ can be tested by computing the corresponding dot product. 
Similarly, the intersection $\bI$ of two lines $\ell_1$ and $\ell_2$ corresponds to a vector that is orthogonal to both, resulting in $\bI = \ell_1 \times \ell_2$.

To compute the barycentric coordinates of the stated points $\bA_t$, $\bB_t$, $\bC_t$ of the nedian family, we first determine  
$$
\ell_{\bA \bA_t'} = \begin{bmatrix}0\\-1+t\\0\end{bmatrix}, \quad 
\ell_{\bB \bB_t'} = \begin{bmatrix}t\\0\\0\end{bmatrix}, \quad \text{and}\quad
\ell_{\bC \bC_t'} = \begin{bmatrix}-t\\ 1-t\\ 0\end{bmatrix}.
$$
where 
$$
\bA_t' =\begin{bmatrix}t \\0 \\ 1-t\end{bmatrix}, \quad 
\bB_t' = \begin{bmatrix}0\\1-t\\ t\end{bmatrix}, \quad \text{and}\quad \bC_t' = \begin{bmatrix}1-t\\t\\0\end{bmatrix}. 
$$
It follows that the intersections of corresponding lines yield 
\begin{align*}
\bA_t = \ell_{\bB\bB_t'} \times \ell_{\bC\bC_t'} = \begin{bmatrix}(1-t)t\\ t^2\\ (1-t)^2\end{bmatrix},\quad
\bB_t = \ell_{\bC\bC_t'} \times \ell_{\bA\bA_t'} = \begin{bmatrix}(1-t)^2\\(1-t)t\\ t^2\end{bmatrix},\quad
\bC_t = \ell_{\bA\bA_t'} \times \ell_{\bB\bB_t'} = \begin{bmatrix} t^2\\ (1-t)^2\\ (1-t)t\end{bmatrix}.
\end{align*}
\end{proof}

\begin{proof}[Proof of \lemref{lem:inversesbPhi}]
The statements follow from a straightforward computation using \lemref{lem:proptriangleConcatenation}: 

\begin{itemize}
\item The inverse of the family of scaled triangles $\bPhiS$ described in~\teqref{eqn:scaling} is given by
    $$
    \bPhiS^{-1} : t \mapsto \left( 1 + \frac{2}{t}, 1 - \frac{1}{t}, 1 - \frac{1}{t}\right)
    $$
For $t\neq 0$, this inverse is a scaling by $\frac{1}{t}$, 
    $$
    \bPhiS^{-1}(t) = \bPhiS\left(\frac{1}{t}\right).
    $$ 

\item According to the formula in \lemref{lem:proptriangleConcatenation}, we have that
    $$
    \bPhiAinv \colon t \mapsto \left(
    \frac{(1-t)t}{1-3t(1-t)},
    \frac{t^2}{1-3t(1-t)}, 
    \frac{(1-t)^2}{1-3t(1-t)}\right).
    $$

To show that $\bPhiAinv(t)$ is a member of the nedian family, our goal is to find a parameter $t_N\in\RR$ that would satisfy $\bPhiAinv(t) = \bPhiN(t_N)$. Consequently, for $i\in \{1,2,3\}$, we consider 
    \begin{align*}
        \frac{\left(\bPhiAinv\right)_{i}(t)}
        {\left(\bPhiAinv\right)_{1}(t) + \left(\bPhiAinv\right)_{2}(t)+ \left(\bPhiAinv\right)_{3}(t)} &= 
        \frac{\left(\bPhiN\right)_{i}(t_N)}
        {\left(\bPhiN\right)_{1}(t_N) + \left(\bPhiN\right)_{2}(t_N)+ \left(\bPhiN\right)_{3}(t_N)}.
    \end{align*}
    These three equations simplify to 
    \begin{align*}
         \frac{(1-t)t}{1-3t(1-t)} &= \frac{(1-\tN)\tN}{1-(1-\tN)\tN}, \\
         \frac{t^2}{1-3t(1-t)} &= \frac{\tN^2}{1-(1-\tN)\tN}, \\
         \frac{(1-t)^2}{1-3t(1-t)} &= \frac{(1-\tN)^2}{1-(1-\tN)\tN}.
    \end{align*}
    While all three equations are quadratic in $\tN$, for $t\neq \nicefrac{1}{2}$, they share a common solution, namely
    $$
    \tN = -\frac{t}{1 - 2t}.
    $$

\item   According to the formula in \lemref{lem:proptriangleConcatenation}, it holds that 
    $$
    \bPhiNinv \colon t \mapsto \left(
    0, \frac{1-t}{(1-2t)(1-(1-t)t)}, -\frac{t}{(1-2t)(1-(1-t)t}\right).
    $$
    Since the nedian triangle family only degenerates to a point for \(t=\nicefrac{1}{2}\) and \((t^2-t+1) > 0\), the claim follows after dividing by the common factors \((2t-1)(t^2-t+1)\neq 0\) due to the homogeneity of barycentric coordinates. 

  Similarly to the other case, to show that $\bPhiN^{-1}(t)$ is a member of the aliquot family, our goal is to find a parameter $\tA\in\RR$ that would satisfy $\bPhiNinv(t) = \bPhiA(\tA)$. Consequently, for $i\in \{1,2,3\}$, we consider 
    \begin{align*}
        \frac{\left(\bPhiNinv\right)_{i}(t)}
        {\left(\bPhiNinv\right)_{1}(t) + \left(\bPhiNinv\right)_{2}(t)+ \left(\bPhiNinv\right)_{3}(t)} &= 
        \frac{\left(\bPhiA\right)_{i}(\tA)}
        {\left(\bPhiA\right)_{1}(\tA) + \left(\bPhiA\right)_{2}(\tA)+ \left(\bPhiA\right)_{3}(\tA)}.
    \end{align*}
    These three equations simplify to 
    \begin{align*} 
    0 = 0, \quad
    \frac{(1-t)}{1-2t} = 1-\tA, \quad
    -\frac{t}{1-2t} = \tA.
    \end{align*}
    For $t \neq \nicefrac{1}{2}$, both equations are solved by
    $$
    \tA = -\frac{t}{1 - 2t}.
    $$
\end{itemize}

\end{proof}

\end{document}